\documentclass[a4paper,11pt]{article}
\usepackage{amsthm,amsmath,amssymb}
\usepackage{mathrsfs}
\usepackage[dvipdfmx]{graphicx}
\usepackage{here}
\usepackage{color}
\usepackage{braket}
\usepackage{bm}
\usepackage[
 bookmarks=true,%
 bookmarksnumbered=true,%
 colorlinks=true,%
 linkcolor=blue,%
 setpagesize=false,%
 pdftitle={title},%
 pdfauthor={Ryoji Takano},%
 pdfsubject={2nd},%
 pdfkeywords={Rough path}]{hyperref}
\usepackage{color}
\usepackage{braket}
\usepackage{bm}

\newcommand{\bbN}{\mathbb{N}}
\newcommand{\bbR}{\mathbb{R}} 
\newcommand{\bbQ}{\mathbb{Q}}

\newcommand{\EXP}[1]{\mathsf{E}\left[#1\right]}
\newcommand{\PROB}[1]{\mathsf{P}\left[#1\right]}

\newcommand{\hd}{H\" older }

\newcommand{\D}{\mathrm{d}}
\newcommand{\Hd}{H\"{o}lder}
\newcommand{\mcK}{\mathcal{K}}

\usepackage{pb-diagram}

\usepackage[textsize=tiny]{todonotes}

\numberwithin{equation}{section}

\theoremstyle{plain}
\theoremstyle{definition}
\newtheorem{definition}{Definition}[section]
\newtheorem{theorem}[definition]{Theorem}
\newtheorem{proposition}[definition]{Proposition}
\newtheorem{corollary}[definition]{Corollary}
\newtheorem{lemma}[definition]{Lemma}
\newtheorem{remark}[definition]{Remark}
\newtheorem{assumption}[definition]{Hypothesis}

\usepackage[utf8]{inputenc}

\title{LDP0322}
\author{Ryoji Takano}
\date{March 2022}

\title{Large deviation principle for stochastic differential equations driven by stochastic integrals}
\author{Ryoji Takano \thanks{
Graduate School of Engineering Science, Osaka University, 
1-3, Machikaneyama, Toyonaka, Osaka, 560-8531, Japan. 
Email: {\tt rtakano@sigmath.es.osaka-u.ac.jp}
}}
\date{}
\begin{document}
\maketitle

\begin{abstract}
In this paper, we prove the pathwise large deviation principle (LDP) for stochastic differential equations driven by stochastic integrals in one dimension. 
The result can be proved with a minimal use of rough path theory, and this implies the pathwise LDP for many class of rough volatility models, and it characterizes the asymptotic behavior of implied volatility.
First, we introduce a new concept called $\alpha$-Uniformly Exponentially Tightness, and prove the pathwise LDP for stochastic integrals on \Hd \ spaces. 
Second, we apply this type of LDP to deduce the pathwise LDP for stochastic differential equations driven by stochastic integrals in one dimension.
Finally, we derive the asymptotic behavior of implied volatility as an application of main results.
\end{abstract}

\section{Introduction}
A rough volatility model is a stochastic volatility model for an asset price process in which the H\"{o}lder regularity of volatility processes is less than half. 
In recent years, such a model has attracted attention, because as shown by \cite{F21}, rough volatility models are the only class of continuous price models that are consistent to a power law of implied volatility term structure typically observed in equity option markets. 
Proving a large deviation principle (LDP) is one way to derive the power law under rough volatility models as done by many authors using various methods \cite{FoZh,BaFrGuHoSt, BaFrGaMaSt,  FrGaPi1, FrGaPi2,Gu, GuViZh1,GuViZh2, JaPaSt, GeJaPaStWa, JaPa, Gu22,FuTa}.
An introduction to LDP and some of its applications to finance and insurance problems are discussed in \cite{Pham, FGGJT}.
One precise approximation formula for implied volatility is the BBF formula~\cite{BBF, ABBF}, which follows from short-time LDP under local volatility models. 
On the other hand, the SABR formula, which is of daily use in financial practice, is also proved for a valid approximation under the SABR model by means of LDP~\cite{Osa}.
From these relations between LDP and precise approximation under classical (non-rough) volatility models, we expect LDP for rough volatility models to provide in particular a useful implied volatility approximation formula for financial practice such as model calibration.

For the proof for pathwise LDP of standard stochastic differential equations (SDEs), an elegant method using rough path theory was proposed \cite{FV,FV10}.
The continuity of the solution map on rough path spaces is key to derive the pathwise LDP for such SDEs.
However, when considering rough volatility models, the usual rough path theory does not work because the regularity of the volatility process is lower than that of asset prices, and so stochastic integrands are not controlled by the stochastic integrators in the sense of  \cite{Gubinelli}. 
Nevertheless, methods which are analogue of rough path theory have been proposed to prove the pathwise LDP for rough volatility model, one uses the theory of regularity structure \cite{BaFrGaMaSt}, and another uses a variant of rough path theory \cite{FuTa}. 
In \cite{JaPaSt, BaFrGaMaSt}, the following It\^{o} SDE is discussed (here $Y$ represents the dynamics of the logarithm of a stock price process):
\begin{align*}
     \D Y_t =  f(\hat{X}_t,t) \D X_t -\frac{1}{2} f^2(\hat{X}_t,t) \D t,
\end{align*}
where $X$ is a Brownian motion, $\hat{X}$ is the Riemann-Liouville type fractional Brownian motion with Hurst index $H \in (0,1/2)$, and $f$ is a smooth function. 
This SDE is called rough Bergomi model \cite{BaFrGa}. In \cite{BaFrGaMaSt}, the authors proved the short-time LDP for rough Bergomi models and by using the continuity of Hairer's reconstruction map. 
The point is that its proof comes down to the small-noise LDP for ``models" which construct the solution of the rough Bergomi model. 
On the other hand, this result was extended to situations where rough volatility models have local volatility in \cite{FuTa};
\begin{align}\label{model}
    \D Y =  \sigma(Y) f(\hat{X},t) \D X -\frac{1}{2} \sigma^2(Y) f^2(\hat{X},t) \D t,
\end{align}
where $X$ is a Brownian motion, $\hat{X} := \int_0^{\cdot} \kappa (\cdot-s)\D W_s$ (where $\kappa$ is a deterministic singular kernel and $W$ is a Brownian motion), $\sigma $ and  $f$ are smooth functions respectively. 
If $\sigma =1$ and $\kappa = \kappa_H$ ($\kappa_H$ is the Riemann-Liouville kernel, see (\ref{RLk})), \eqref{main} is the  rough Bergomi model. 
In \cite{FuTa}, partial rough path spaces lacking the iterated integral of $\hat{X}$ were considered, and a partial rough path integration map  was constructed. 
By using the continuity property of this integration map, the small-noise and short-time LDP for (\ref{model}) were proved based on the pathwise LDP for the canonical noises constructed by $(X,\hat{X})$ on partial rough path spaces. 
Compared with \cite{BaFrGaMaSt}, the framework of \cite{FuTa} is more elementary and one can prove that not only the LDP for rough Bergomi model but also that for many rough volatility models, see the lists of Introduction in \cite{FuTa}. 
However, the continuity property of the integration map in \cite{FuTa} relies on the smoothness of the coefficient $f$, because the higher order Taylor expansion of $f$ is needed to cover the low regularity of $\hat{X}$. 
For these reason, although the previous work \cite{FuTa} is widely applicable, it goes beyond the framework of it when $f$ is not smooth. 
For example generalized rough volatility models discussed in \cite{HoJaMu} or when $f(y,t) := \sqrt{y}$ \cite{FoZh}.

Inspired by above previous research,  we will discuss the following SDE in one dimension in this paper:
\begin{align}\label{main}
 \D Y^\epsilon = \sigma(Y^\epsilon) A^\epsilon_t \D X^\epsilon - \frac{1}{2} \sigma^2(Y^\epsilon) (A^\epsilon)^2_t \D t
\end{align}
Here  $X^\epsilon := \epsilon^{1/2} X$,   $X$ is a one dimensional Brownian motion,  $A^\epsilon$ is an adapted continuous process and $\sigma$ is a smooth function.
If $A^\epsilon=f(\hat{X}^\epsilon,\cdot)$, \eqref{main} coincides with \eqref{model}.
In this paper, we will discuss the pathwise LDP for  \eqref{main}.

Now we consider how to prove the pathwise LDP for (\ref{main}). 
Let $A\cdot X$ be the It\^{o} stochastic integral for $A$ with respect to $X$ and $\Lambda(t) := t$.
Let also $ Z^\epsilon: = (A^\epsilon\cdot X^\epsilon, (A^\epsilon)^2\cdot \Lambda)$ and we will regard $Z^\epsilon$ as the driver for \eqref{main}.
Then we define the Young pairing (see Section 9.4 in \cite{FV10}) $\mathbb{Z}^\epsilon$ for $Z^\epsilon$ and we regard $\mathbb{Z}^\epsilon$ as the canonical lift for rough path spaces. 
Since $X^\epsilon$ are one dimensional paths, the mapping $Z^\epsilon \mapsto \mathbb{Z}^\epsilon$ is continuous. 
Combining to the usual rough path theory, we finally can construct the solution $Y^\epsilon$ of \eqref{main} from $Z^\epsilon$:
\[
  \begin{diagram}
   \node{G \Omega_{\alpha\text{-Hld}} }     \node{ \mathbb{Z}^\epsilon} \arrow{e,t}{\text{sol. map} } \node{ \mathbb{Y}^\epsilon } \arrow{s,r}{\text{projection}} \node{\quad } \\
   \node{C^{\alpha \text{-Hld}} }  \node{Z^\epsilon}  \arrow{n,l}{\text{Young pair}} \node{Y^\epsilon} \node{\quad}
  \end{diagram}
\]
Here for $\alpha \in (1/3,1/2]$, $C^{\alpha \text{-Hld}}$ is H\"{o}lder spaces, and $G \Omega_{\alpha\text{-Hld}}$ is rough path spaces, and ``sol. map" in the above diagram means the solution map in the sense of rough differential equations.
Therefore, the pathwise LDP for \eqref{main} can be proved from the small noise LDP for $\{Z^\epsilon\}_{\epsilon>0}$ on H\"{o}lder spaces.
This idea enables us to avoid adherence to use the smoothness of coefficient $f$ which is the essential condition to cover the low regularity of $\hat{X}$ in \cite{BaFrGaMaSt, FuTa}. 
We also note that our approach does not use a variant of rough path theory or regularity structure theory, which means we are able to obtain simpler proof. 
Although the small noise LDP for stochastic integrals with respect to the uniform topology was proved in \cite{Ga08}, this results cannot be applied for our methods, because our idea requires the small noise LDP for stochastic integrals with respect to ``\Hd \ topology''.

Our method also allows for a unified treatment of pathwise LDP for rough volatility models, compared with \cite{BaFrGaMaSt, FuTa}. 
For example, the pathwise LDP for rough volatility models were discussed under the different assumptions which are not mutually inclusive \cite{FoZh,HoJaMu,GuJaRoSh,FuTa,BaFrGaMaSt}, but these results indeed are included in our setting. 
To the best of the author's knowledge, no such pathwise LDP for these models is known in the literature.

In the perspective of applications for mathematical finance, it is important to derive the asymptotic formula of the implied volatility because of the pricing of put/call options. 
Moreover, the formula is applicable to check whether models are consistent to the power law of implied volatility or not.
For example, generalized rough volatility models discussed in \cite{HoJaMu} are widely applicable, in the sense that the authors of \cite{HoJaMu} provide us how to make a
 numerical approximation of such models.
 Although one reason for using and studying such models is that it is expected to be consistent with the power laws of the implied volatility observed in the market, there is no justifications of this expectation in the literature.
As an application of our analysis, we will prove the short-time LDP of them (actually one can treat more general models) and derive an asymptotic formula of the implied volatility which tells us the models are consistent to the power law of the implied volatility (Corollary \ref{IVS}). 
This formula is described as a generalization of Forde and Zhang's work \cite{FoZh}.
 
In section 2, we first discuss the pathwise LDP for stochastic integrals on \Hd \ spaces, this is Theorem \ref{main1}. 
We next discuss the pathwise LDP for (\ref{main}) (actually we will discuss the Stratonovich SDEs \eqref{SDEs} corresponding to \eqref{main}), see Theorem \ref{main4}.
In section 3, we will first show how to apply main results to the pathwise LDP for rough volatility models, see Theorem \ref{main5} for small-noise LDP and see Theorem  \ref{shorttime} for short-time LDP. 
Then we will derive an asymptotic formula for implied volatility, see Corollary \ref{IVS}.
In section 4, we will prove the main theorem in order.
  
\section{Main results}

 \subsection{Large deviation principle for stochastic integrals}
 We first review the LDP for stochastic integrals with respect to the uniform topology discussed in \cite{Ga08}.
 In \cite{Ga08}, the index set of a sequence of stochastic processes is the natural number set $n \in\mathbb{N}$, but regarding as $n= \epsilon^{-1}$, we consider the family of stochastic processes  $\{X^{\epsilon}\}_{\epsilon>0}$ which means that the index set of it is $(0,1]$.
 We say that a real function $x: [0,\infty) \to \bbR$ is cadlag if $x$ is continuous on the right and has limits on the left.
 Throughout this paper, we fix a filtered probability space satisfying the usual conditions $(\Omega, \mathcal{F}, (\mathcal{F}_t)_{t \ge 0}, \mathsf{P})$ .

 \begin{definition}[Definition 1.1 \cite{Ga08}]\label{UUET}
 Let $\{ X^{\epsilon}\}_{\epsilon > 0} $ be a family of real valued cadlag semi-martingales. We say that the family $\{ X^{\epsilon} \}_{\epsilon > 0}$ is uniformly exponentially tight with speed $\epsilon^{-1}$ if for every $t > 0 $ and every $M > 0$, there is $K_{M,t} > 0$ such that
 \begin{align}\label{UET}
     \limsup_{\epsilon \searrow 0 } \epsilon \log \sup_{U \in \mathcal{S}} \PROB{\sup_{s \le t } | (U_{-} \cdot X^{\epsilon})_s | \ge K_{M,t} }   \le -M,
 \end{align}
 where $\mathcal{S}$ be the set of all simple adapted processes $U$ with $\sup_{t \ge0} |U_t| \le1$ and $(U_{-})_t := \lim_{s \to t-} U_s$. 
 In this paper, denote $U \cdot X$ by the stochastic integral for $U$ with respect to a semi-martingale $X$ in It\^{o} sense:
 \begin{align*}
     (U\cdot X)_t: = \int_0^t U_r \D X_r .
 \end{align*}
 \end{definition}

\begin{definition}[Section 1.2 \cite{DeZe}]
 Let $(E,\mathcal{B}(E))$ be a metric space with a Borel $\sigma$-algebra $\mathcal{B}(E)$. 
\begin{enumerate}
\item[$(i)$] We say that a function $I:E \rightarrow [0,\infty]$ be a good rate function if, for all $\lambda \in [0,\infty)$, the set 
\[
\{ x \in E : I(x) \le \lambda \},
\] 
is compact on $E$.
\item[$(ii)$] We say that the family of measures $\{\mu_{\epsilon} \}_{\epsilon > 0}$ on $E$ satisfies the LDP with speed $\epsilon^{-1}$ with good rate function $I$ if, for all $\Gamma \in \mathcal{B}(E)$, 
\[
- \inf_{x \in \Gamma^{o} } I(x) 
\le \liminf_{\epsilon \searrow 0} \epsilon \log \mu_{\epsilon} (\Gamma) 
\le \limsup_{\epsilon \searrow 0} \epsilon \log \mu_{\epsilon} (\Gamma)
\le - \inf_{x \in \overline{\Gamma} } I(x),
\] 
where $\overline{\Gamma}$ is the closure of $\Gamma$, and  $\Gamma^o$ is the interior of $\Gamma$.
\end{enumerate}
\end{definition}

In this paper, `` with speed $\epsilon^{-1}$'' is omitted. Denote by  $ D([0,\infty),\mathbb{R})$ the space of all cadlag functions and denote by $d_D$  the Skorohod topology (see Chapter 3 in \cite{Bill}). 

 \begin{lemma}[Theorem 1.2 \cite{Ga08}]\label{Ga}
Let $\{X^{\epsilon} \}_{\epsilon > 0}$  be a uniformly exponentially tight family of cadlag adapted semi-martingales on $\bbR$ and $\{ A^{\epsilon} \}_{\epsilon > 0}$ be a family of real valued cadlag adapted processes.
Assume that $ \{ (A^{\epsilon},X^{\epsilon}) \}_{\epsilon > 0}$ satisfies the LDP on $(D([0,\infty),\mathbb{R}) ,d_D) \times (D([0,\infty),\mathbb{R}) ,d_D)$  with good rate function $I^{\#}$. 
Then the family of triples $\{ (A^{\epsilon},X^{\epsilon},A^{\epsilon}\cdot X^{\epsilon}) \}_{\epsilon > 0}$ satisfies the LDP on $(D([0,\infty),\mathbb{R}) ,d_D) \times (D([0,\infty),\mathbb{R}) ,d_D) \times (D([0,\infty),\mathbb{R}) ,d_D) $ with good rate function
\[
I(a,x,z) =
\begin{cases}
    I^{\#}(a,x),  & z = a \cdot x, \ x \in \textbf{BV},\\
    +\infty, & \text{otherwise},
\end{cases}
\]
where $\textbf{BV}$ is the set of bounded variation and $a\cdot x$ means the Riemann-Stieltjes integral for $a$ with respect to $x$.
 \end{lemma}
 
We wii improve on this result in terms of \Hd\ topology.
For $\alpha \in (0,1]$, denote by $C^{\alpha \text{-Hld}}([0,1],\bbR)$  the \hd space with the \hd norm
 \[
\|x\|_{\alpha\text{-Hld} } := |x_0| +  \sup_{0\le s<t \le 1} \frac{|x_t-x_s|}{|t-s|^{\alpha}},
\]
and let 
\begin{align*}
    C_{0}^{\alpha \text{-Hld}}([0,1],\bbR)
    := \{ x \in C^{\alpha\text{-Hld}} ([0,1],\bbR) : \lim_{\delta\searrow0}w_{\alpha}(\delta,x) =0 \},
\end{align*}
where 
\begin{align*}
    w_{\alpha}(\delta,x) := \sup_{|t-s|\le \delta} \frac{|x_t-x_s|}{|t-s|^{\alpha}}.
\end{align*}
Note that $C_{0}^{\alpha\text{-Hld}}([0,1],\bbR)$ is a separable Banach space, see \cite{Ha}.
We next introduce a concept of $\alpha$-Uniformly Exponentially Tight.
 
 \begin{definition}\label{uet2}
We fix $\alpha \in(0,1]$. Let $\{ X^{\epsilon} \}_{\epsilon > 0}$ be a family of real valued continuous semi-martingales on $[0,1]$. We say that the family $\{ X^{\epsilon} \}_{\epsilon > 0}$ is \emph{ $\alpha$-Uniformly Exponentially Tight} if, for all $M>0$, there exists $K_M >0$ such that
\begin{equation}\label{1.1}
    \limsup_{\epsilon\searrow 0} \epsilon \log \sup_{U\in \mathcal{B} ([0,1],\mathbb{R})}  
    \PROB{ \| U\cdot X^{\epsilon} \|_{\alpha \text{-Hld}} \ge K_M } \le - M,
\end{equation}
where $\mathcal{B}([0,1],\mathbb{R} )$ is the set of all adapted, left continuous with right limits  processes $U$ on $[0,1]$ such that $\sup_{t\in[0,1]} |U_t| \le 1$.
\end{definition}

\begin{remark}
$\alpha$-Uniformly Exponentially Tight is stronger than uniformly exponentially tight in the following sense. Assume that $\{ X^{\epsilon} \}_{\epsilon > 0}$ is $\alpha$-Uniformly Exponentially Tight.
Note that for all $U \in \mathcal{S}$, $U_{-} \in \mathcal{B}([0,1],\bbR)$. For all $M>0$, take $K_M>0$ such that (\ref{1.1}) holds. Then we have that for all $t \in (0,1)$,
\begin{align*}
     &\limsup_{\epsilon \searrow 0 } \epsilon \log \sup_{U \in \mathcal{S}} \PROB{  \sup_{s \le t } |(U_{-} \cdot X^{\epsilon} )_s | \ge K_{M}} \\
     & \le \limsup_{\epsilon\searrow 0} \epsilon \log \sup_{\tilde{U} \in \mathcal{B} ([0,1],\mathbb{R})}  
     \PROB{ \| \tilde{U} \cdot X^{\epsilon} \|_{\alpha \text{-Hld}} \ge K_M}\\
     &\le - M.
\end{align*}
Hence we conclude that $\{ X^{\epsilon} \}_{\epsilon > 0}$ satisfies (\ref{UET}) when $t \in (0,1)$.
\end{remark}

Let $C([0,1],\bbR)$ be the set of all real valued continuous functions on $[0,1]$ equipped with the uniform topology. 
Here, we state our first main result, the proof is given in Section \ref{sec31}.
\begin{theorem}\label{main1}
We fix  $ \alpha \in (0,1]$ and $\beta <\alpha$.
Let $\{ X^{\epsilon} \}_{\epsilon > 0}$ be a family of real valued $\alpha$-\Hd \ continuous semi-martingales on $[0,1]$ and $\{ A^{\epsilon} \}_{\epsilon > 0}$ be a family of real valued adapted continuous processes on $[0,1]$ such that $A^{\epsilon}\cdot X^{\epsilon} \in C^{\alpha\text{-Hld}}([0,1],\bbR)$.
Assume that  $ \{ (A^{\epsilon} ,X^{\epsilon}) \}_{\epsilon > 0}$ satisfies the LDP on $C([0,1],\bbR) \times C^{\alpha\text{-Hld}}_0([0,1], \mathbb{R})$ with good rate function $I^{\#}$.

Then if $\{X^{\epsilon}\}_{\epsilon > 0}$ is $\alpha$-Uniformly Exponentially Tight, $\{ (A^{\epsilon},X^{\epsilon} ,A^{\epsilon} \cdot X^{\epsilon} )\}_{\epsilon > 0}$ satisfies the LDP on $ C ([0,1], \mathbb{R}) \times C_0^{\alpha \text{-Hld}}([0,1], \mathbb{R}) \times C_0^{\beta \text{-Hld}}  ([0,1], \mathbb{R})$ with good rate function $I$;
\[
I(a,x,z) = \begin{cases}
    I^{\#}(a,x), & z = a \cdot x, \ x \in \textbf{BV}, \\
    + \infty, & \text{otherwise}.
\end{cases}
\]
\end{theorem}

\begin{remark}
    Let $V$ is an adapted continuous process.
    Note that if  $V \in C^{\alpha\text{-Hld}}([0,1],\bbR)$ and $\beta<\alpha$, then $V \in C^{\beta\text{-Hld}}_0([0,1],\bbR)$. 
    Note also that since  $V$ is an adapted continuous process, we have that 
    \[
    \|V\|_{\beta \text{-Hld}} = \sup_{0\le s < t \le1 , s,t \in \bbQ} \frac{|V_t -V_s|}{|t-s|^\beta}
    \] 
    is $\mathcal{F} /\mathcal{B} (\bbR)$-measurable.
    Since $C^{\beta\text{-Hld}}_0([0,1],\bbR)$ is a separable Banach space, we conclude that $V$ is $\mathcal{F}/\mathcal{B} (C^{\beta\text{-Hld}}_0([0,1],\bbR))$-measurable. 
\end{remark}

One of the most important family of $\alpha$-Uniformly Exponentially Tight semi-martingales is constructed from scaled Brownian motions. 
The proof is deferred to Section \ref{sec31}.

\begin{proposition}\label{main2}
We fix $  \alpha \in [1/3, 1/2)$. 
Let $B$ be an $\bbR$-valued  standard Brownian motion on $[0,\infty)$ and assume that $B$ is $(\mathcal{F}_t)$-adapted. Let $B^{\epsilon} := \epsilon^{1/2} B$ and $\bar{B}^{\epsilon} := B_{\epsilon \cdot}$. 
Then we have that:

\begin{enumerate}
    \item[(i)] for all $(\mathcal{F}_t)$-adapted continuous processes $A$ on $[0,1]$, we have $A\cdot B^{\epsilon} \in C^{\alpha\text{-Hld}}([0,1],\bbR)$, and  $\{B^{\epsilon}\}_{\epsilon>0}$ is  $\alpha$-Uniformly Exponentially Tight: for all $M>0$, there exists $K_M>0$ such that
\begin{equation*}
 \limsup_{\epsilon \searrow 0 } \epsilon \log  \sup_{U \in \mathcal{B} ([0,1],\bbR ) }  \PROB{\|U \cdot B^{\epsilon} \|_{\alpha  \text{-Hld}} \ge K_M } \le - M.
\end{equation*}
    \item[(ii)] Let $\mathcal{F}^\epsilon_t := \mathcal{F}_{\epsilon t}$. Then for all $(\mathcal{F}^\epsilon_t)$-adapted continuous processes $\bar{A}$ on $[0,1]$, we have $\bar{A} \cdot \bar{B}^{\epsilon} \in C^{\alpha\text{-Hld}}([0,1],\bbR)$ and $\{ \bar{B}^{\epsilon}\}_{\epsilon>0}$ is  $\alpha$-Uniformly Exponentially Tight.
\end{enumerate}
\end{proposition}

\subsection{Large deviation principle for stochastic differential equations driven by stochastic integrals in one dimension}

In this section, we will discuss how to derive the LDP for SDEs driven by stochastic integrals in one dimension from Theorem \ref{main1}. 
Consider the Stratonovich SDEs in one dimension:
\begin{align}\label{SDEs}
\D Y_t =    \sigma_1(Y_t ) \circ A_t \D X_t + \sigma_2 (Y_t) \tilde{A}_t \D t, \quad  Y_0 \in \bbR,
\end{align}
where $\sigma_1 , \sigma_2 \in C^3_b$, $X$ is a one-dimensional standard  Brownian motion, and $A$ and $\tilde{A}$ are real valued adapted continuous processes respectively.
Note that we regard \eqref{SDEs} as the equation driven by a stochastic integral $A \cdot X$ and $\tilde{A}\cdot \Lambda$ where $\Lambda(t) :=t$.

Let 
 \begin{align*}
      Z= (Z^{(1)},Z^{(2)}) := ( A \cdot  X, \tilde{A} \cdot \Lambda ),
 \end{align*}
where $\cdot$ means the It\^{o} integral. Let also that 
\begin{align}\label{ZZ}
    \mathbb{Z}_{st}   : = (1,Z_{st},\mathbf{Z}_{st}), \quad 0\le s <t \le 1,
\end{align}
where for $i,j \in \{1,2\}$,
\begin{align*}
    Z_{st} := Z_t -Z_s, \quad  
    \mathbf{Z}^{(ij)}_{st} :=
    \begin{cases}
        2^{-1} (Z^{(1)}_{st})^2, & i=j=1,\\
        \int_s^t (Z^{(i)}_r - Z^{(i)}_s)  \D Z^{(j)}_r, & \text{otherwise},
    \end{cases}
\end{align*}
 and  $\mathbf{Z}$ is defined by the Young integral (see also section 9.4 in \cite{FV10}, this is the Young pairing).
 Note that by Proposition \ref{main2} (i), for $\alpha \in (1/3,1/2)$, $Z \in C^{\alpha\text{-Hld}} \times  C^{1\text{-Hld}}$ and so the Young integral is well-defined.
For $\alpha \in (1/3,1/2]$, denote by  $G \Omega^{\alpha \text{-Hld}} ([0,1],\mathbb{R}^2)$ the geometric rough path space and  $d_{\alpha}$ the metric function on $G \Omega^{\alpha \text{-Hld}} ([0,1],\mathbb{R}^2)$ (see Section 2.2 in \cite{FH}). 
One can prove that for $\alpha \in (1/3,1/2)$, $\mathbb{Z} \in G \Omega^{\alpha \text{-Hld}} ([0,1],\mathbb{R}^2)$, see the proof of Theorem \ref{main4}.

We now discuss the following type of rough differential equation (RDE) (in Lyons' sense; see Section~8.8 of \cite{FH}, for example):
\begin{equation}\label{RDE}
 \bar{Y}_t =  \int_0^t \bar{\sigma}(\bar{Y}_u) \D \mathbb{Z}_u,
\end{equation}
where $\bar{Y}_t=Y_t-Y_0$, $\bar{\sigma}(y)= ( \sigma_1(Y_0 + y) , \sigma_2(Y_0 + y))$.

\begin{theorem} \label{thm:33}
Let $\sigma_1 , \sigma_2 \in C^{3}_b$.
\begin{enumerate}
    \item[(i)]
 RDE $(\ref{RDE})$ driven by \eqref{ZZ} has a unique solution $ \bar{Y} = \Phi(\mathbb{Z},y_0)$, where
\begin{equation*}
\Phi : \Omega_{\alpha \text{-Hld}} ([0,T],\mathbb{R}^2) \times \mathbb{R} \rightarrow C^{\alpha \text{-Hld}} ([0,T],\mathbb{R})
\end{equation*}
is the solution map of \eqref{RDE} that is locally Lipschitz continuous with respect to $d_{\alpha}$.
\item[(ii)]
The  first level of the last component $\bar{Y}$ of the solution to RDE \eqref{RDE} for \eqref{ZZ}  gives the solution
$Y(\omega) = y_0 + \bar{Y}$ to the stratonovich SDE \eqref{SDEs}.
\end{enumerate}
\end{theorem}
\begin{proof}
 These are standard results from rough path theory; see e.g., Theorem~1 in \cite{Le} or Chapter~8 in \cite{FH} for $(i)$ and Chapter~9 in \cite{FH} or Theorem~17.3 in \cite{FV10} for $(ii)$.
\end{proof}

\begin{remark}
Although the solution $\bar{Y}$ is one-dimension, the noise $Z$ is a two dimensional path and so it is not trivial whether $\bar{Y}$ can be constructed from $Z$ or not, and this is why we need to consider rough paths $\mathbb{Z}$ of $Z$.
\end{remark}

 Let $X^{\epsilon} := \epsilon^{1/2} X$, and $A^{\epsilon}, \tilde{A}^{\epsilon}$ are $(\mathcal{F}_t)$ adapted continuous processes respectively (these correspond  scaled processes of  $A,\tilde{A}$ respectively).
 Let
 \[
 Z^{\epsilon} = ((Z^{(1)})^{\epsilon},(Z^{(2)})^{\epsilon}) := ( A^{\epsilon} \cdot  X^{\epsilon}, \tilde{A}^{\epsilon} \cdot \Lambda ),
 \] 
and we define $\mathbb{Z}^{\epsilon}$ like (\ref{ZZ}).
 We now consider the following scaled SDEs:
\begin{align}\label{strat1}
\D Y^{\epsilon}_t = \sigma_1(Y^{\epsilon}_t ) \circ  A^{\epsilon}_t \D X^{\epsilon}_t 
+ \sigma_2 (Y^{\epsilon}_t) \tilde{A}^{\epsilon}_t \D t,
\end{align}
We state the second main result, the proof is given in Section \ref{sec32}.
\begin{theorem}\label{main4}
 We fix $\alpha \in [1/3, 1/2)$. 
 Assume that  there exists $\alpha' \in [1/3,1/2)$ with $\alpha' > \alpha$ such that $\{(A^{\epsilon}, \tilde{A}^{\epsilon},X^{\epsilon} ) \}_{\epsilon > 0}$ satisfies the LDP on $C ([0,1], \mathbb{R})\times C ([0,1], \mathbb{R}) \times C_0^{\alpha'\text{-Hld}} ([0,1], \mathbb{R})$ with good rate function $J^{\#}$. 
 
 Then $\{ Y^{\epsilon} \}_{\epsilon > 0}$ satisfies the LDP on $C^{\alpha \text{-Hld}} ([0,1] , \mathbb{R})$ with good rate function
 \begin{align*}
  J (y) := \inf \left \{ J^{\#} (a,\tilde{a},x) :  y = \Phi \circ F ( a \cdot x, \tilde{a} \cdot \Lambda), \  x \in \textbf{BV} \right \},
 \end{align*}
 where 
 \begin{align}\label{defF}
     F(z)_{st}  := (1,z_{st}, \mathbf{z}_{st}),
 \end{align}
 and for $i,j \in \{1,2\}$,
\begin{align*}
    z^{(i)}_{st} := z^{(i)}_t -z^{(i)}_s,\quad 
\mathbf{z}_{st}^{(ij)} := 
\begin{cases}
       2^{-1}(z_{st}^{(i)})^2, & i=j=1,\\
        \int_s^t ( z_r^{(i)} - z_s^{(i)}) \D z_r^{(j)}, & \text{otherwise},
\end{cases}
\end{align*}
and  $\mathbf{z}$ is defined by the Young integral.
\end{theorem}

\section{An application for mathematical finance}\label{finance}

\subsection{Small noise asymptotics for rough volatility models \eqref{rSABR}}
We now discuss an application of Theorem \ref{main4}. Let $\kappa :(0,1] \rightarrow [0 , \infty) $ as
\begin{align*}
    \kappa (t) := g(t) t^{\gamma - \alpha}, \quad t \in (0,1], 
\end{align*}
where $\alpha,\gamma \in(0,1)$ and $g$ is a Lipschitz function.
Let $\mathcal{K}: C^{\alpha \text{-Hld}} ([0,1],\bbR)\rightarrow C^{\gamma \text{-Hld}}([0,1],\bbR)$ as
\begin{align*}
   \mathcal{K} f(t) 
& := \lim_{\epsilon \searrow 0} \left \{ \left[\kappa(t-\cdot) (f(\cdot) - f(t)) \right]^{t-\epsilon}_0 
+ \int_0^{t-\epsilon} (f(s) - f(t)) \kappa'(t-s) \D s \right\}\\
& = \kappa(t) (f(t) - f(0)) + \int_0^t (f(s)- f(t) ) \kappa' (t-s) \D s.
\end{align*} 
This map is called the fractional integral for $\gamma > \alpha$ and the fractional derivative for $\gamma < \alpha$, see  \cite{Fu23} for details. 
For simplicity, let $\mu : = \gamma - \alpha$.

\begin{remark}\label{kernel}
Because of the existence of Lipschitz part $g$, $\kappa$ has sufficient generality for applications. For example, we can take the following singular kernels.
\begin{enumerate}
\item the Riemann--Liouville kernel 
\begin{align}\label{RLk}
      \kappa_H (t) := t^{H-1/2}, \quad  t \in (0,1], \ H \in (0,1/2)
\end{align}
has the above form ($\mu =  H - 1/2$). 
\item the Gamma fractional 
\begin{align*}
      \kappa(t) := t^{\mu } \exp{(ct)}, \quad  t \in (0,1], \  \mu \in (-1,1), \ c <0,
\end{align*}
\item Power-law 
\begin{align*}
      \kappa(t) := t^{\mu} (1+t)^{\beta - \mu}, \quad  t \in (0,1], \ \mu \in (-1,1), \ \beta <-1.
\end{align*}
\end{enumerate}
For convenience, we denote $\mathcal{K}_0$ by $\mathcal{K}$ associated with the Riemann--Liouville kernel   $\kappa_H$, which means $\mathcal{K}_0$ is the usual fractional operator. 
\end{remark}

We fix $\alpha \in (0,1/2)$ and $\gamma \in (0,1)$ ($\alpha$ and $\gamma$ are the parameters of $\mathcal{K}$ respectively). 
Denote by $(W,W^{\perp})$ a two-dimensional standard Brownian motion.
Set
\begin{align}\label{defX}
    X := \rho W + \sqrt{1-\rho^2} W^{\perp},\quad  V := \Psi ( \mathcal{K} A) , \quad \rho \in [-1,1],
\end{align}
where  $A$ is the solution to the SDE
\begin{align*}
    \D A_t = b(A_t) \D t + a(A_t) \D W_t, \quad A_0 \in \bbR,
\end{align*}
$a,b \in C^4_b$, and  $\Psi: \bbR \to  \bbR $ is a nice function (see in Remark \ref{co}). 
Consider the following It\^{o} SDEs (here $Y$ represents the dynamics of the logarithm of a stock price process):
\begin{align}\label{rSABR}
\D Y_t =   \sigma(Y_t ) f(V_t ,t) \D X_t  -\frac{1}{2} \sigma^2 (Y_t) f^2(V_t, t) \D t,\quad Y_0 \in \mathbb{R}
\end{align}
where $f : \mathbb{R} \times [0,1] \rightarrow [0,\infty)$ be a nice function (see in Remark \ref{co}), and $\sigma : \mathbb{R} \rightarrow \mathbb{R}$ is  in $C^4_b$. 
In this paper, we call \eqref{rSABR} rough volatility models.
The equation can be rewrite in the Stratonovich sense:
\begin{align*}
\D Y_t = \sigma(Y_t ) \circ f(V_t ,t)  \D X_t
-\frac{1}{2} \left \{  \sigma^2 (Y_t) +  \sigma (Y_t) \sigma' (Y_t) \right \}  f^2(V_t, t) \D t 
\end{align*}
Note that we regard this SDE as the equation driven by a stochastic integral $f(V ,\cdot) \cdot X$ and $f^2(V,\cdot) \cdot \Lambda$.
For $\epsilon >0$, let $(X^{\epsilon},V^{\epsilon}) := (\epsilon^{1/2} X, \epsilon^{1/2} V)$ and consider the following SDEs:
\begin{align}\label{esrSABR}
\D Y^{\epsilon}_t  
= \sigma(Y^{\epsilon}_t   ) \circ f( V^{\epsilon}_t  ,t)  \D X^{\epsilon}_t 
 -\frac{1}{2} \left \{  \sigma^2 (Y^{\epsilon}_t ) 
+  \sigma (Y^{\epsilon}_t ) \sigma' (Y^{\epsilon}_t ) \right \}  f^2(V^{\epsilon}_t , t) \D t , \quad \epsilon >0,
\end{align}

Let 
 \[
 Z^{\epsilon} := (f(V^{\epsilon} ,\cdot) \cdot X^{\epsilon}, f^2(V^{\epsilon} , \cdot) \cdot \Lambda ), \quad \epsilon >0.
 \]

We state an application of Theorem \ref{main4} for rough volatility models, the proof is given in Section \ref{sec32}.

\begin{theorem}\label{main5}
 We fix $\alpha \in [1/3, 1/2)$ and $\gamma \in(0,1)$. 
 Assume that $x \mapsto f(x,\cdot)$ is continuous map on $C ([0,1], \mathbb{R})$, $x \mapsto \Psi (x)$ continuous map from $C^{\gamma\text{-Hld}}_0 ([0,1], \mathbb{R})$ into $C ([0,1], \mathbb{R})$  and let $Y^{\epsilon}$ is the solution of (\ref{esrSABR}).  
 Then $\{ Y^{\epsilon} \}_{\epsilon > 0}$ satisfies the LDP on $C^{\alpha \text{-Hld}} ([0,1] , \mathbb{R})$ with good rate function
 \begin{align*}
 \tilde{J}(y) := \inf \left \{ \frac{1}{2} \| (w,w^\perp) \|^2_{\mathcal{H}}  :   y = \Phi \circ F \circ  F_f \circ \mathbb{K} (w,w^\perp) ,\ (w,w^\perp) \in \mathcal{H}\right \},
 \end{align*}
 where $\mathcal{H}$ is the Cameron-Martin space on $\bbR^2$,
 \begin{align}\label{defK}
     \mathbb{K} w : = ( \Psi \mathcal{K} (a(A_0) w^{(1)} ), \rho w^{(1)} + \sqrt{1-\rho^2}w^{(2)}),
     \quad F_f(v,x) :=  \left ( f(v, \cdot ) \cdot x, f(v , \cdot )^2 \cdot \Lambda  \right),
 \end{align}
 $F$ is defined as \eqref{defX}.
\end{theorem}

\begin{remark}[assumptions for $f$ and $\Psi$]\label{co}
If $f : \bbR \times [0,\infty) \to \bbR$ is in $C^1$-class or $f(x,t)$ does not depend on $t \in [0,\infty)$ and locally $\beta$-H\"{o}lder continuous with respect to $x \in  \bbR$ ($\beta \in (0,1]$), then $f$ satisfies the assumption of Theorem. 
Similarly, if $\Psi : \bbR \to \bbR $ is $\beta$-\Hd \ continuous, then  $\Psi$ satisfies the assumption of Theorem.
As will be discussed later (Remark \ref{comp} and Table \ref{table1}), these sufficient assumptions are weaker than these in previous works \cite{BaFrGaMaSt, FuTa, FoZh, JaPaSt}.
\end{remark}

\begin{remark}[comparison with previous studies : small-noise LDP]\label{comp1}
Although there are few previous works about small-noise LDP for rough volatility models when compered with short-time LDP, this Theorem is a natural extension of the previous work \cite{FuTa}. 
In \cite{FuTa}, the authors discussed in the case of generalized rough SABR models (the case when $f \in C^\infty$,  $\sigma \in C^4_b$, $\Psi := \text{id}$, and $A$ is a Brownian motion).
The main difference is that the proof of Theorem \ref{main5} is much simpler than that of \cite{FuTa}, because our method only uses the standard rough path theory while the method of \cite{FuTa} need to use a partial rough path theory  which is further developments of it.
Also, our theorem is more flexible than the previous result in terms of $f$, $\mathcal{K}$, $A$, and $\Psi$.
For example, one can derives the small-noise LDP for generalized rough Heston models discussed in \cite{HoJaMu} (see also Remark \ref{comp}).
\end{remark}

\subsection{Short time asymptotics for rough volatility models \eqref{rSABR}}

In this subsection, we prove the LDP for $\{ t^{\mu} (Y_t -Y_0)\}_{t>0}$ on $\bbR$ when $t \searrow 0$, where $Y$ is the solution for \eqref{rSABR}, $\mu := \gamma -\alpha $, and $\gamma ,\alpha$ are the parameter for a kernel $\kappa$. 
To do so, let 
\[
\tilde{Y}^\epsilon_t :=\epsilon^{\mu} (Y_{\epsilon t} - Y_0).
\]
 Note that 
\begin{align*}
V_{\epsilon t} = \Psi ( \mathcal{K} A_{\epsilon t} ) = \Psi ( \mathcal{K}^{\epsilon} (\epsilon^\mu (A_{\epsilon \cdot} -A_0 )_t )  = :V^\epsilon_t,
\end{align*}
where
\begin{align*}
  \mathcal{K}^\epsilon f(t) 
& := \epsilon^{-\mu} \left[    \kappa_{\epsilon} (t) (f(t) - f(0)) +  \int_0^t (f(s)- f(t) ) \frac{\D}{\D t} \kappa_{\epsilon} (t-s) \D s \right], \quad \epsilon >0,
\end{align*} 
and $\kappa_{\epsilon} (t) := \kappa (\epsilon t)$. 
Here we use the relation $\mathcal{K} A (\epsilon t) = \mathcal{K}^\epsilon ( \epsilon^{\mu} A_{\epsilon \cdot})$ in the second equality. 
By using the  change of variables for stochastic integrals and Riemann integrals, one has that
\begin{align*}
    \tilde{Y}^\epsilon_t 
    & = \epsilon^{\mu} \left \{  \int_0^{\epsilon t} \sigma (Y_u) f(V_u,u) \D X_u - \frac{1}{2}  \int_0^{\epsilon t} \sigma^2(Y_u) f^2(V_u,u) \D u   \right\}\\
    & =  \int_0^t \tilde{\sigma}^\epsilon(\tilde{Y}^\epsilon_u) \D (Z^{(1)})^\epsilon_u -\frac{1}{2} \int_0^t (\tilde{\sigma}^\epsilon)^2(\tilde{Y}^\epsilon_u)\D (Z^{(2)})^\epsilon_u,
\end{align*}
where
\begin{align*}
(Z^{(1)})^{\epsilon}_t & :=   \int_0^{ t}  f(V^\epsilon_u, \epsilon u) \D ( \epsilon^{\mu} X_{\epsilon u} ),\\
(Z^{(2)})^{\epsilon}_t & := \epsilon^{\mu +1}   \int_0^t f^2(V^{\epsilon}_u, \epsilon u) \D u,
\end{align*}
and 
\begin{equation*}
   \tilde{\sigma}^\epsilon(s) := \sigma(Y_0 + \epsilon^{-\mu}s).
\end{equation*}
Note that $Z^{(1)}$ is well-defined since $V^\epsilon$ and $X_{\epsilon \cdot}$ are $( \mathcal{F}^\epsilon_t) := (\mathcal{F}_{\epsilon t})$-adapted respectively.
Indeed, we can derive an LDP for $\{ \tilde{Y}^\epsilon \}$ under the following Hypothesis \ref{hypo}. 

\begin{assumption}\label{hypo}
We assume that $\kappa$ satisfies the following conditions:
\begin{enumerate}
\item $\alpha \in (0,1/2)$, and $\mu <0$, and $\sigma \in C^4_b$,
\item the Lipschitz part $g$ of $\kappa$ is in $C^2_b$, and  $\sup_{t \in [0,1] } |g(\epsilon t) -1| \to 0$ as $\epsilon \searrow 0$.
\item $f :\bbR \times [0,\infty) \to \bbR$ is a continuous function and satisfies the following conditions: for all $v^n ,v \in C([0,1],\bbR)$ with $v^n \to v$ in $C([0,1],\bbR)$, 
\[
\sup_{\epsilon \in (0,1)} \sup_{t \in [0,1]} |f(v_t,\epsilon  t) - f(v^n_t,\epsilon  t) | \to 0, \quad \text{as $n \to \infty$}.
\] 
\[
\sup_{t \in [0,1] } |f(v_t,\epsilon t) - f(v_t, 0) | \to 0 , \quad \text{as $\epsilon \searrow 0$.}
 \]
\end{enumerate}
\end{assumption}

Note that the assumption 2 is harmless in the sense that all examples which appear in Remark \ref{kernel} satisfy them. 
Note also that the assumption 3 is harmless in the sense that the all functions discussed in Remark \ref{co} satisfy this condition.
\begin{theorem}\label{short}
Assume that Hypothesis \ref{hypo}.  

Then $\{ \tilde{Y}^{\epsilon} \}_{0 < \epsilon \leqq 1}$ satisfies the LDP on $C^{\alpha\text{-Hld}} ([0,1] ,\bbR)$ as $\epsilon \searrow 0$ with speed $\epsilon^{-(2\mu +1)}$ with good rate function
 \begin{align*}
 \tilde{J} (\tilde{y}) 
:= \inf \left \{ \frac{1}{2} \| (w,w^\perp  ) \|^2_{\mathcal{H}} :
  \begin{aligned}
  & x = \rho  w + \sqrt{1 - \rho^2} w^\perp , \\
    & \tilde{y} = \sigma(y_0) \int_0^{\cdot} f( \Psi  \mathcal{K}_0 (a (A_0) w )_r ,0)  \D x_r , \  (w,w^\perp ) \in \mathcal{H}
      \end{aligned}  \right\},
\end{align*}
\end{theorem}
The proof of Theorem \ref{short}  is given in Section \ref{shortproof}.

An LDP for the marginal distribution $\tilde{Y}^\epsilon_1$ follows from the contraction principle, and the corresponding one-dimensional rate function as follows.
\begin{theorem}\label{shorttime}
Assume Hypothesis \ref{hypo}.
Then  $\{ t^{\mu} (Y_t -Y_0)\}_{ 1\ge t>0}$ satisfies the LDP on $\bbR$ with $t \searrow 0$ with speed $t^{- (2\mu +1)}$ with good rate function
\begin{align}\label{ratef}
\tilde{J}^{\#} (z) 
:= \inf_{g \in L^2([0,1])} 
\left [ \frac{1}{2} \int_0^1 |g_r|^2 \D r + \frac{ \left \{z - \rho \sigma(Y_0) \int_0^1 f \left ( v(g)_r,0 \right) g_r \D r \right\}^2 }{2 (1 - \rho^2) \sigma(Y_0)^2 \int_0^1 f\left(v(g)_r ,0 \right)^2 \D r } \right],
\end{align}
where $ v(g)  = a (A_0) \Psi ( \mathcal{K}_0 g)$, and $\mathcal{K}_0 g := \int_0^{\cdot} \kappa_H(t-r) g_r \D r$.
\end{theorem}

\begin{proof}
By the contraction principle and the previous theorem, $\{ t^{\mu} (Y_t -Y_0)\}_{ 1\ge t>0}$ satisfies the LDP on $\bbR$ with $t \searrow 0$ with speed $t^{-(2\mu +1)}$ with good rate function
 \begin{align*}
 \tilde{J}^{\#} ( z) 
:= \inf \left \{ \frac{1}{2} \| (w ,w^\perp ) \|_{\mathcal{H}}^2 :
  \begin{aligned}
  & x = \rho  w + \sqrt{1 - \rho^2} w^\perp ,\\
    & z = \sigma(y_0) \int_0^{1} f(\Psi \mathcal{K}_0 (a (A_0) w )_t,0) \D x_r , \  (w,w^\perp ) \in \mathcal{H} 
      \end{aligned}  \right\}.
\end{align*} 
By using the argument in Theorem 3.8 in \cite{FuTa}, one can prove that $ \tilde{J}^{\#} $ has the above representation \eqref{ratef}. 
\end{proof}

\begin{remark}[comparison with previous studies : short-time LDP]\label{comp}
The theorem is a natural extension for the results in \cite{FoZh,JaPaSt, BaFrGaMaSt,FuTa} because if $\Psi = \text{id}$, $a =1$, and $b=0$, then the statement (in particular, the rate function \eqref{ratef}) corresponds to that appeared in the previous works.
First we will compare the assumption for the parameter of  models (see Table \ref{table1}). 
In view of the local volatility function $\sigma$, our method outperforms \cite{FoZh,JaPaSt,BaFrGaMaSt}, although we have to assume a slightly stronger smoothness than \cite{FuTa} because we transform \eqref{rSABR} into stratonovich SDEs.
On the other hand, our assumption of the stochastic volatility function $f$ is the most general in the sense of Remark \ref{co}. 
Also, our method is more flexible than the others in the sense of the fractional operator $\mathcal{K}$, $A$, and $\Psi$. 
Our method allows us to add a Lipschitz part of $\kappa$, a diffusion process $A$, and some transformation $\Psi$ of volatility processes.

As mentioned Remark \ref{comp1}, our first contribution is that the proof is much simpler than that of previous works, because our method only uses the standard rough path theory, while the method of \cite{BaFrGaMaSt, FuTa} need to use regularity structures or a partial rough path theory which are further developments of it (\cite{FoZh,JaPaSt}  is somewhat less applicable).
Moreover, our method allows for a unified treatment of short-time LDP for rough volatility models. 
For example, the result obtained by \cite{FuTa} does not contain the Forde \& Zhang'  result \cite{FoZh} since \Hd \ continuous function is not smooth in general. 

Furthermore, one can derives the short-time LDP for generalized rough Heston models discussed in \cite{HoJaMu}:
\begin{align}\label{Heston}
    \D Y_t & = - \frac{1}{2} V_t\D t + \sqrt{V_t} \D X_t,\quad  Y_0=0,\\
    V_t &= \tilde{\Psi} (\mathcal{K} A)_t \notag,
\end{align}
where $x \mapsto \tilde{\Psi} (x)$ is a continuous map from $C^{\gamma \text{-Hld}} ([0,1])$ into $C_+^{\gamma \text{-Hld}} ([0,1]) $.
The equation \eqref{Heston} coincides with the equation \eqref{rSABR} with $\sigma=1$ and $f (v,t) :=\sqrt{v}$. 

\eqref{Heston} are widely applicable, in the sense that the authors of \cite{HoJaMu} provide us how to make a numerical approximation of the solution of \eqref{Heston}.
Although a reason for using and studying such models is that it is expected to be consistent with the power laws of implied volatility observed in the market, there is no justifications of this expectation in the literature because $f$ is not smooth and $\mathcal{K}$, $A$, or $\Psi$ are general.  
One can remedy this problem in the sense that the approximation formula \eqref{concl} which is given later and consistent to the power law of the implied volatility in the market is obtained.
\end{remark}

\begin{table}
\begin{center}
\caption{Short time asymptotics for rough volatility models \eqref{rSABR}}\label{table1}
\begin{tabular}{cccccc} \hline
   method & $\sigma$ & $f$ & $\mathcal{K}$ & $A$ & $\Psi$ \\ \hline
   Forde  \& Zhang \cite{FoZh}& 1 & \Hd \ continuous  & $\mathcal{K}_0$ & $a=1$, $b=0$ & id  \\
   Jacquier et al. \cite{JaPaSt} & 1 &  $\sqrt{\exp{(x - t^{2H} /2 ) }}$  &  $\mathcal{K}_0$ & $a=1,b=0$ & Id\\ 
   Bayer et al. \cite{BaFrGaMaSt}& 1 & $C^\infty$ & $\mathcal{K}_0$ & $a=1$, $b=0$ & id  \\ 
   Fukasawa  \& T \cite{FuTa} & $C^3_b$ & $C^\infty$ & $\mathcal{K}_0$ & $a=1$, $b=0$ & id  \\ \hline
 Our method & $C^4_b$ & general (see Remark \ref{co})  & general & general & general  \\ \hline
\end{tabular}
\end{center}
\end{table}

\begin{remark}[flexibility of the parameters of \eqref{rSABR}]\label{depend}
Although we adopt the generalized fractional operator $\mathcal{K}$, it is $\mathcal{K}_0$ that appears in the rate function \eqref{ratef} for the short time asymptotics of \eqref{rSABR} and so the effect of $\mathcal{K}$ is partial in this sense. In other words, the generality of Lipschitz part $g$ of $\kappa$ does not affect it.
On the other hand, $\Psi$ and $a(A_0)$ truly affects the rate function for the short time asymptotics of \eqref{rSABR}.
Also the scaling order for $Y$ does depends on the parameter $\mu$, and these suggests that $\mu$ does affect the implied volatility skew.
\end{remark}

\subsection{Short time asymptotics for put/call options and implied volatility for rough volatility models \eqref{rSABR}}

Let 
\begin{align*}
\Lambda^\ast(x) := 
\begin{cases}
\inf_{y>x}  \tilde{J}^{\#}  (y), & x \ge 0,\\
\inf_{y<x}  \tilde{J}^{\#}  (y), & x \le 0.
\end{cases}
\end{align*}
where $ \tilde{J}^{\#} $ is defined in Theorem \ref{shorttime}. Let $(x)_{+} := x \vee 0$. 

\begin{theorem}\label{putcall}
Under the Hypothesis \ref{hypo}, we have the following:
\begin{enumerate}
\item[(i)] we have the following small-time behavior for out of the money put option on $S_t = \exp{(Y_t)}$ with $S_0 =1$:
\begin{align*}
 - \lim_{t \searrow 0} t^{2(\mu +1/2)} \log \EXP{(\exp{(xt^{-\mu} )} -S_t)_{+} } = \Lambda^\ast(x), \quad x \le 0, 
\end{align*}
where $x := \log K$ is the log moneyness.
\item[(ii)] Moreover, if we assume that
\begin{align}\label{moment}
\limsup_{t\searrow 0} t^{2\mu+1} \log \EXP{S^q_t} =0,\quad q >1,
\end{align} 
then we also have the following small-time behavior for out of the money call option on $S_t = \exp{(Y_t)}$ with $S_0 =1$:
\begin{align*}
 - \lim_{t \searrow 0} t^{2(\mu +1/2)} \log \EXP{(S_t - \exp{(xt^{-\mu} )}  )_{+} } = \Lambda^\ast(x), \quad x \ge 0. 
\end{align*}
\end{enumerate}
\end{theorem}

\begin{proof}
See Appendix \ref{B}
\end{proof}

\begin{remark}
Because it is difficult to check the exponentially integrability of $S$ in general, the assumption \eqref{moment} is natural, see  Assumption 2.4 in \cite{BaFrGuHoSt}, for example. 
\end{remark}

\begin{corollary}\label{IVS}
Denote $\Sigma (x,t)$ by the implied volatility at the log moneyness $x$ and the maturity $t$.  
Then for the rough volatility models \eqref{rSABR}, we have 
\begin{align}\label{concl}
\lim_{t \searrow 0} \Sigma( x t^{-\mu},t ) = \frac{|x|}{\sqrt{2 \Lambda^\ast (x)}}, \quad x <0.
\end{align}
Moreover, if we assume that \eqref{moment}, then 
\begin{align*}
\lim_{t \searrow 0} \Sigma( x t^{-\mu},t ) = \frac{|x|}{\sqrt{2 \Lambda^\ast (x)}}, \quad x >0.
\end{align*}
\end{corollary}

\begin{proof}
One can adapt the same  argument as Corollary 4.13 in \cite{FoZh}.
\end{proof}

\begin{remark}[observation and future works for \eqref{concl}]
The corollary is an extension for Corollary 4.15 in \cite{FoZh}, and our result outperforms the previous results in the sense that $\sigma$, $f$, $\mathcal{K}$ and $A$ is more general.
The dependence of these parameters is determined by how the rate function $\tilde{J}^\#$  depends on them (see Remark \ref{depend}).
Note that if $-\mu$ is negative, then the steepness of the implied volatility smile is infinite V-shape as $t \searrow 0$, while it is flat when $-\mu$ is positive.  

Although justifications are left in the future, I think there are several chances to apply the asymptotic formula \eqref{concl} for practical applications.
For one thing, the approximate formula \eqref{concl} has a slightly different structure compared to previous studies \cite{FoZh} because of the generality of \eqref{rSABR}.
As a result, the right-hand side of \eqref{concl} may be explicitly solved. 
Even if it cannot be explicitly solved, further precise approximation formula may be found here.
Furthermore, \eqref{concl} is more flexible than that of \cite{FoZh} in terms of the parameters of \eqref{rSABR}. 
This flexibility may have an impact on approximation accuracy when compared with the previous work \cite{FoZh}.
Moreover, the formula may suggest a reasonable and concrete model which is consistent to the power law of the implied volatility.
Finally, by using a numerical implementation method suggested by \cite{FoZh}, one can apply \eqref{concl} to the pricing of put/call options with being consistent to the power law of the implied volatility.
\end{remark}

\section{Proof of main theorems}

\subsection{Proof of Theorem \ref{main1} and \ref{main2}}\label{sec31}

 We fix $\alpha \in (0,1]$ and $\beta < \alpha$. 
 For simplicity, we will write $C_0^{\alpha \text{-Hld}}([0,1],\bbR)$ as $C_0^{\alpha \text{-Hld}}([0,1])$, and $\{ X^{\epsilon} \}_{\epsilon > 0}$ as $\{ X^{\epsilon} \}_{\epsilon}$. 
 
\begin{definition}\label{defG}
For $\delta >0$, we define the map $G_{\delta} : C([0,1]) \times  C_0^{\alpha \text{-Hld}}([0,1])  \to C_0^{\beta \text{-Hld}}([0,1]) $ as follows:
\begin{align*}
G_{\delta} (a,x)_t 
&:=  \sum_{k=1}^{\infty} a_{\tau_{k-1}^{\delta}} \left (  x_{ t \wedge \tau_k^{\delta}}  - x_{ t \wedge \tau_{k-1}^{\delta}}  \right), \quad t \in[0,1],
\end{align*}
where $\tau_0^{\delta} =0$ and
\[
\tau_k^{\delta} = \tau_k^{\delta} (a) := \inf \{t>\tau_{k-1}^{\delta} : |a_t - a_{\tau_{k-1}^{\delta}} | > \delta \} \wedge 1, \quad k \in \mathbb{N}.
\]
\end{definition}
 
\begin{remark}\label{3.4}
   We fix $a\in C([0,1])$. 
   By the definition of $\{\tau_k^{\delta}(a)\}_k$, for all $k \in \bbN$, $\tau_k^{\delta}(a) < \tau_{k+1}^{\delta}(a)$ or $\tau_k^{\delta}(a)=1$. 
   Moreover, there exists $k\in \mathbb{N}$ such that $\tau_k^{\delta}(a) =1$.
\end{remark}

\begin{remark}\label{welldef}
We fix $a \in C([0,1])$ and $\delta>0$. Let $k'_0$ is the smallest integer such that $\tau_{k'_0}^{\delta}(a) =1$. 
For all $0 \le s < t \le 1$, take the smallest number $l,l' \in \mathbb{N}$ such that $s \in [\tau^{\delta}_l, \tau^{\delta}_{l+1} ]$ and $t\in [\tau^{\delta}_{l'}, \tau^{\delta}_{l'+1}]$ ($l \le l'\le k'_0-1$). 
Because
\begin{align*} 
 \left | G_{\delta} (a,x)_t - G_{\delta} (a,x)_s \right| & \le  \|a\|_{\infty} \left \{  \sum_{k=l+2}^{l'} | x_{\tau^{\delta}_k} - x_{ \tau^{\delta}_{k-1}} | 
+  | x_t - x_{\tau^{\delta}_{l'}} | + | x_{\tau^{\delta}_{l+1}} - x_{s} |  \right \}\\
& \le \|a\|_{\infty} \|x\|_{\alpha\text{-Hld}} k'_0  |t-s|^{\alpha},
\end{align*}
 we conclude that $G_{\delta} (a,x)$ belongs to $C^{\alpha \text{-Hld}}([0,1])$. 
 Since $\beta < \alpha$, we have that $G_{\delta} (a,x) \in C_0^{\beta \text{-Hld}}([0,1])$.
\end{remark}

\begin{lemma}\label{num}
We fix $\delta>0$. For $a(n), a \in C([0,1])$, assume that $a(n) \to a$ in $C([0,1])$. 
Then there exists a subsequence $\{a(n')\}_{n'}$ of $\{a(n)\}_n$ and a sequence $\{r^{\delta}_k\}_k$ on $[0,1]$ such that the following properties hold;
\begin{enumerate}
\item[(i)] for all $k \ge1$, $\tau^{\delta}_k(a(n')) \to r^{\delta}_k$ as $n' \to \infty$, and \ $\{r^{\delta}_k\}_k$ is non-decreasing,
\item[(ii)] for all $k\ge 1$, either $r^{\delta}_k < \tau^{\delta}_k(a(n'))$ for all $n'\ge k$, or $r^{\delta}_k \ge \tau^{\delta}_k(a(n'))$ for all $n' \ge k$,
\item[(iii)] for all $k\ge1$, $a(n')_{\tau^{\delta}_k(a(n'))} \to a_{r^{\delta}_k}$ as $n' \to \infty$,
\item[(iv)]  for all $k\ge1$,  $r^{\delta}_k =1$ or  $r^{\delta}_{k-1} < r^{\delta}_k$,
\item[(v)] there exists $k' \in \mathbb{N}$ such that $r^{\delta}_{k'}=1$.
\end{enumerate}
\end{lemma}

\begin{proof}
We adapt the argument of Theorem 6.5 in \cite{Ga08}. We fix $\delta>0$. 
For brevity, we write $\tau^{\delta}_k$ as $\tau_k$ (we apply the same notation to $r^{\delta}_k$).  
Since $\{\tau_1(a(n))\}_{n}$ is a sequence on $[0,1]$, there exists a subsequence $\{a(n^{(1)})\}_{n^{(1)}}$ of $\{a(n)\}_n$ and $r_1 \in [0,1]$ such that $\tau_1(a(n^{(1)})) \to r_1$ as $n^{(1)} \to \infty$. 
Since $\{\tau_2(a(n^{(1)}))\}_{n^{(1)}}$ is a sequence on $[0,1]$, there exists a subsequence $\{a(n^{(2)})\}_{n^{(2)}}$ of $\{a(n^{(1)})\}_{n^{(1)}}$ and $r_2 \in [0,1]$ such that $\tau_2(a(n^{(2)})) \to r_2$ as $n^{(2)} \to \infty$. 
By using the same argument, for all $k\in \mathbb{N}$, there exists a subsequence
\[
\{a(n^{(k)})\}_{n^{(k)}} \subset \{a(n^{(k-1)})\}_{n^{(k-1)}} \subset ... \subset \{a(n^{(1)})\}_{n^{(1)}} \subset \{a(n)\}_n
\]
and $r_k \in [0,1]$ such that $\tau_k (a(n^{(k)})) \to r_k$ as $n^{(k)}\to\infty$. 
Let $a(n'_k) := a(n^{(k)}_k)$ (here $n'_k$ means the $k$-th number of $n'$). 
Then we have that $\{a(n')\}_{n'}$ is a subsequence of $\{a(n)\}_n$, and for all $k\ge1$, $\{a(n'_{j})\}_{j\ge k} \subset \{a(n^{(k)})\}_{n^{(k)}}$. 
So we have that for $k \ge1$, $\tau_k(a(n'))\to r_k$ as $n' \to \infty$.
Since $\tau_{k-1}(a(n')) \le \tau_k (a(n'))$, we have $r_{k-1} \le r_k$. 
So this is a subsequence $\{a(n')\}_{n'}$ of $\{a(n)\}_n$ such that $\{a(n')\}_{n'}$ satisfies $(i)$.
We can also select a subsequence $\{a(n'')\}_{n''}$ of $\{a(n')\}_{n'}$ satisfies $(ii)$. We fix this subsequence $\{a(n'')\}_{n''}$ and rewrite $\{a(n'')\}_{n''}$ as $\{a(n')\}_{n'}$ for brevity.

It is straightforward to show $(iii)$ because of the fact that $\tau_{k} (a(n')) \to r_k$ and the uniform convergence of $\{a(n')\}_{n'}$. 

Let us verify the property $(iv)$. 
If this were not true, there exists $k \in \bbN$ such that $r = r_{k-1} = r_{k} <1$. 
From $(iii)$, we have
\[
    a(n')_{\tau_k(a(n'))}  \to a_r, \quad a(n')_{\tau_{k-1}(a(n')) } \to a_r, \quad \text{as $n' \to \infty$}.
\]
In particular, there exists $N'(\delta) \in \mathbb{N}$ such that if $n' \ge N'(\delta)$,
\[
    |a(n')_{\tau_k(a(n'))} - a(n')_{\tau_{k-1}(a(n'))} | <\delta.
\]
On the other hand, by using $(ii)$, we can prove that there exists $N''(k)\ge1$ such that if $n' \ge N''(k)$, $\tau_{k-1} (a(n'))<\tau_k(a(n'))$.
This is because in the case of $\tau_{k-1}(a(n')) < r_{k-1}$, $r_{k-1} <1$ implies that $\tau_{k-1}(a(n')) <1$ and so by Remark \ref{3.4}, we have that $\tau_{k-1} (a(n'))<\tau_k(a(n'))$. 
In the case of $\tau_{k-1}(a(n')) \ge r_{k-1}$, since $r_{k-1} <1$, there exists $N''(k-1)$ such that if $n' \ge N''(k-1)$,  $\tau_{k-1}(a(n')) <1$ and so $\tau_{k-1} (a(n'))<\tau_k(a(n'))$.
Then the definition of $\tau_k$ implies that if $n' \ge N''(k)$,
\[
\delta \le |a(n')_{\tau_k(a(n'))} -a(n')_{\tau_{k-1}(a(n')) } |,
\]
 and this is a contradiction.

It remains to verify $(v)$. 
If this were not true,  for all $k\ge 1$, $r_k<1$. 
By using the same argument in the proof of $(iv)$, we can prove that for all $k \ge1$, there exists $N'(k)\ge1$ such that if $n' \ge N''(k)$, $\tau_{k-1} (a(n'))<\tau_k(a(n'))$. 
Then by the definition of $\tau_k$,  if $n \ge N''(k)$,
\[
\delta \le |a(n')_{\tau_k(a(n'))} - a(n')_{\tau_{k-1}(a(n'))}|.
\]
Since $r_k <1$, $(iv)$ implies that $r_{r-1}<r_k$. Then $(iii)$ implies that
\[
\delta \le |a_{r_k} - a_{r_{k-1}}|.
\]
On the other hand, since $\{r_k\}_k$ is a non-decreasing and bounded sequence, there exists $R \in [0,1]$ such that $r_k \to R$ as $k \to \infty$ and so the continuity of $a$ implies that $a_{r_k} \to a_R$ as $k \to \infty$. In particular, $\{a_{r_k}\}$ is a Cauchy sequence. However, this is the contradiction.
\end{proof}

\begin{definition}[Definition 6.1 \cite{Ga08}]\label{almost compact}
Let $E_1,E_2$ be a metric space respectively. We say that a function $G:E_1 \rightarrow E_2$ is almost compact if for all $x \in E_1$ and  $\{x(n)\}_n \subset E_1$ with $x(n) \to x$ in $E_1$, there exists a subsequence $\{x(n_k) \}_k$ and $y \in E_2$ such that $G(x(n_k)) \to y$ in $E_2$.  
\end{definition}

\begin{lemma}\label{G}
For all $\delta>0$, $G_{\delta}$ is almost compact.
\end{lemma}

\begin{proof}
We fix $\delta>0$. 
Assume that $(a(n),x(n)) \to (a,x) $ in $ C([0,1]) \times C_0^{\alpha \text{-Hld}}([0,1])$.
Take a subsequence $\{a(n')\}$ of $\{a(n)\}$, and $\{r_k\}_k$ such that the properties of Lemma \ref{num} hold. 
Let $k_0$ be the smallest number such that $(v)$ holds in Lemma \ref{num} and let
\[
z_t  := \sum_{k=1}^{\infty} a_{r_{k-1}} \left( x_{t \wedge r_k} - x_{t \wedge r_{k-1}}\right), \quad t \in [0,1].
\]
By using the same argument as $G_\delta (a,x)$ in Remark \ref{welldef}, one can show that 
\begin{align*}
 |z_t - z_s  |   \le k_0 \|a\|_{\infty} \|x\|_{\alpha \text{-Hld}} |t-s|^{\alpha},
\end{align*}
and this implies that $z \in C_0^{\beta \text{-Hld}}([0,1])$.

We will show that $G_{\delta}(a(n'),x(n')) \to z$ in $C_0^{\beta \text{-Hld}}([0,1])$. 
We fix $\eta>0$ and we will prove that there exists $N(\eta) \in \bbN$ such that if $n \ge N(\eta)$,
\begin{align*}
\|G_\delta( a(n'),x(n')) - z \|_{\beta\text{-Hld}} \lesssim \eta.
\end{align*}
To show the assertion, we fix $0 \le s < t \le1$ and take the smallest number $l',l \in \mathbb{N}$  such that $s \in [r_{l'},r_{l'+1}]$ and $t \in [r_{l} ,r_{l+1}]$ ($l' \le l \le k_0-1$).

Let $\tilde{\delta} = \tilde{\delta} (k_0,\eta) := (\min_{0\le k \le k_0} |r_k -r_{k-1}| ) \wedge 1$ and $N(\eta) := \max_{0 \le k \le k_0} N(k,\eta)$, where $N(k,\eta)$ is a number such that if $n' \ge N(\eta,k)$ : for $0 \le k \le k_0$,
\begin{enumerate}
    \item[$(I)$]   $|\tau_k(a(n')) - r_k| < \tilde{\delta}/2$,
    \item[$(II)$]  $|a(n')_{\tau_k(a(n'))} - a_{r_k}| < \eta$,
    \item[$(III)$]  $\|a(n')- a\|_{\infty} < \eta$ and $\|x(n')- x\|_{\infty} < \eta$.
\end{enumerate}
Now we fix $n' \ge N(\eta)$. For brevity, we write $\tau_k(a(n'))$ as $\tau_k$.
By $(I)$, we can prove that $r_{k+1}<\tau_{k+2}$ and $\tau_{k-1}<r_k$ for all $0\le k \le k_0$. 
Hence we consider the following nine cases.
\begin{enumerate}
\item[(Case1)] $s\in[\tau_{l'},\tau_{l'+1}]$ and $t\in[\tau_{l},\tau_{l+1}]$.
\begin{align*}
    &|G_{\delta}(a(n'),x(n'))_t - G_{\delta}(a(n'),x(n'))_s - (z_t - z_s)|\\ 
    & \le |a(n')_{\tau_l} (x(n')_t - x(n')_{\tau_l} ) - a_{r_l} (x_t - x_{r_l})| \\
    & \quad + \sum_{k=l'+2}^l |a(n')_{\tau_{k-1}} (x(n')_{\tau_k} - x(n')_{\tau_{k-1}} ) - a_{r_{k-1}} (x_{r_k} - x_{r_{k-1}})| \\
    & \quad + |a(n')_{\tau_{l'}} (x(n')_{\tau_{l'+1}} - x(n')_s ) - a_{r_{l'}} (x_{r_{l'+1}} - x_s)| =: T_{11}+T_{12}+T_{13}.
\end{align*}
Since $|\tau_l - r_l| < |t-s|$, $|\tau_l - t| < |t-s|$, and $|t - r_l| < |t-s|$, $(I)$ to $(III)$ imply that 
\begin{align*}
    T_{11} 
    & \le |a(n')_{\tau_l} - a_{r_l}| |x(n')_t - x(n')_{\tau_l}| + |a_{r_l}| |x(n')_t - x(n')_{\tau_l} - (x_t - x_{r_l})|\\
    & \le |a(n')_{\tau_l} - a_{r_l}| \|x(n')\|_{\alpha\text{-Hld}}  |t - \tau_l|^{\alpha}\\
    & \quad + \|a\|_{\infty} \left\{  |x(n')_t - x(n')_{r_l} -( x_t  - x_{r_l})| 
    + |x(n')_{r_l} - x(n')_{\tau_l} | \right\}\\
    & \le \Big [  |a(n')_{\tau_l} - a_{r_l}|\|x(n')\|_{\alpha\text{-Hld}}
    + \|a\|_{\infty} \|x(n')-x\|_{\alpha\text{-Hld}}  \\
    & \quad + \|a\|_{\infty} \|x(n')\|_{\alpha\text{-Hld}} |\tau_l -r_l|^{\alpha-\beta} \Big] |t-s|^{\beta}\\ 
    & \lesssim (\eta + \eta^{\alpha-\beta})|t-s|^{\beta}.
\end{align*}
Since $|r_k - r_{k-1}|<|t-s|$, $|\tau_k - \tau_{k-1}| < |t-s|$, and $|\tau_k - r_k| < |t-s|$ for all $l'+2 \le k \le l$, $(I)$ to $(III)$ imply that
\begin{align*}
   T_{12} 
    & \le |a(n')_{\tau_{k-1}} - a_{r_{k-1}}| |x(n')_{\tau_k} - x(n')_{\tau_{k-1}} | \\
    &\quad + |a_{r_{k-1}}\|x(n')_{\tau_{k}} - x(n')_{\tau_{k-1}} - (x_{r_k} - x_{r_{k-1}})|\\
    & \le |a(n')_{\tau_{k-1}} - a_{r_{k-1}}|  \|x(n')\|_{\alpha\text{-Hld}}  |\tau_{k} - \tau_{k-1}|^{\alpha} \\
    & \quad + \|a\|_{\infty} \Big \{ |x(n')_{r_k} - x(n')_{r_{k-1}} - (x_{r_k} - x_{r_{k-1}})| \\
    & \quad +  |x(n')_{\tau_k} - x(n')_{\tau_{k-1}} - (x(n')_{r_k} - x(n')_{r_{k-1}})| \Big \}\\
    & \le \Big \{   |a(n')_{\tau_{k-1}} - a_{r_{k-1}} | \|x(n')\|_{\alpha\text{-Hld}}
    + \|a\|_{\infty}\|x(n')-x\|_{\alpha\text{-Hld}} \\
    & \qquad + \|a\|_{\infty} \|x(n')\|_{\alpha\text{-Hld}} \{ |\tau_k -r_k|^{\alpha-\beta} + |\tau_{k-1} -r_{k-1}|^{\alpha-\beta}  \}  \Big \} |t-s|^{\beta}\\
    & \lesssim (\eta + \eta^{\alpha-\beta})|t-s|^\beta,
\end{align*}
and we can estimate $T_{12}$.
Since $|\tau_{l'+1} - s| < |t-s|$, $|r_{l'+1} -s| < |t-s|$, and $|\tau_{l'+1} - r_{l'+1}| < |t-s|$,
\begin{align*}
    T_{13} 
    & \le |a(n')_{\tau_{l'}} - a_{r_{l'}}| |x(n')_{\tau_{l'+1}} - x(n')_s | \\
    & \quad + |a_{r_{l'}}\|x(n')_{\tau_{l'+1}} - x(n')_s - (x_{r_{l'+1}} - x_s)|\\
    & \le |a(n')_{\tau_{l'}} - a_{r_{l'}}| \|x(n')\|_{\alpha \text{-Hld}} |\tau_{l'+1} -s|^{\alpha}\\
    & \quad + \|a\|_{\infty} \Big \{  |x(n')_{r_{l'+1}} - x(n')_s - (x_{r_{l'+1}} - x_s)|  +|x(n')_{r_{l'+1}} - x(n')_{\tau_{l'+1}} | \Big \}\\
    & \le \Big \{  |a(n')_{\tau_{l'}} - a_{r_{l'}}| \|x(n')\|_{\alpha\text{-Hld}} \\
    & \quad + \|a\|_{\infty} \{  \|x(n')-x\|_{\alpha \text{-Hld}}  + \|x(n')\|_{\alpha \text{-Hld}} |r_{l'+1} -\tau_{l'+1}|^{\alpha-\beta} \} \Big \} |t-s|^{\beta}\\
    & \lesssim (\eta + \eta^{\alpha-\beta})|t-s|^{\beta},
\end{align*}
and so by using these inequalities, we have that
\[
    |G_{\delta}(a(n'),x(n'))_t - G_{\delta}(a(n'),x(n'))_s - (z_t - z_s)| \lesssim (\eta + \eta^{\alpha-\beta})|t-s|^\beta.
\]
\item[(Case2)] $s\in[\tau_{l'},\tau_{l'+1}]$ and $t\in[\tau_{l+1},\tau_{l+2}]$.
\begin{align*}
    &|G_{\delta}(a(n'),x(n'))_t - G_{\delta}(a(n'),x(n'))_s - (z_t - z_s)|\\ 
    & \le |a(n')_{\tau_{l+1}}(x(n')_t  - x(n')_{\tau_{l+1}} )| + |a(n')_{\tau_l} (x(n')_{\tau_{l+1}} - x(n')_{\tau_l} ) - a_{r_l} (x_t - x_{r_l})|  \\
    & \quad + \sum_{k=l'+2}^{l} |a(n')_{\tau_{k-1}} (x(n')_{\tau_k} - x(n')_{\tau_{k-1}} ) - a_{r_{k-1}} (x_{r_k} - x_{r_{k-1}})| \\
    & \quad + |a(n')_{\tau_{l'}} (x(n')_{\tau_{l'+1}} - x(n')_s ) - a_{r_{l'}} (x_{r_{l'+1}} - x_s)| =: T_{21}+T_{22}+T_{23}+T_{24}.
\end{align*}
We can estimate $T_{23}$ and $T_{24}$ as the same argument of $T_{12}$ and $T_{13}$ in (Case1). 
To estimate $T_{21}$,  let us note that $(I)$ implies 
\[
     \tau_{l+1}  \le t < r_{l+1} < \tau_{l+1} + \tilde{\delta}/2,
\]
and so $|t-\tau_{l+1}| < \tilde{\delta}/2$. Then $|\tau_{l+1} - t|<|t-s|$ implies that 
\begin{align*}
    T_{21}
      \le \|a(n')\|_{\infty} \|x(n')\|_{\alpha\text{-Hld}} |t - \tau_{l+1}|^{\alpha-\beta}  |t-s|^{\beta}
      \lesssim \eta^{\alpha-\beta} |t-s|^{\beta}.
\end{align*}
Since $|\tau_{l+1} - \tau_l|<|t-s|$, $|r_l - \tau_l|<|t-s|$, and $|\tau_{l+1} - t|<|t-s|$, $(I)$ to $(III)$ imply that
\begin{align*}
    T_{22}
    & \le |a(n')_{\tau_l} - a_{r_l}| |x(n')_{\tau_{l+1}} - x(n')_{\tau_l} |  + |a_{r_l}\|x(n')_{\tau_{l+1}} - x(n')_{\tau_l} - (x_t - x_{r_l})|\\
    & \le  |a(n')_{\tau_l} - a_{r_l}| \|x(n')\|_{\alpha\text{-Hld}} |\tau_{l+1}- \tau_l|^{\alpha}\\
    & \quad + \|a\|_{\infty} \Big  \{ |x(n')_{\tau_{l+1}} -x(n')_{\tau_l} -  ( x_{\tau_{l+1}} - x_{\tau_l})| +  |x_t -x_{r_l} -  ( x_{\tau_{l+1}} - x_{\tau_{l}}) | \Big \}\\
    & \le \Big \{  |a(n')_{\tau_l} - a_{r_l}|  \|x(n')\|_{\alpha\text{-Hld}}\\
    & \quad + \|a\|_{\infty} \{ \|x(n')-x\|_{\alpha\text{-Hld}}  + \|x\|_{\alpha\text{-Hld}} \{ |t- \tau_{l+1} |^{\alpha-\beta}+|\tau_l- r_l |^{\alpha-\beta}\} \}  \Big \} |t-s|^{\beta}\\
    &\lesssim (\eta + \eta^{\alpha-\beta})|t-s|^{\beta}.
\end{align*}

\item[(Case3)] $s\in[\tau_{l'},\tau_{l'+1}]$ and $t\in[\tau_{l-1},\tau_{l}]$.
\begin{align*}
    &|G_{\delta}(a(n'),x(n'))_t - G_{\delta}(a(n'),x(n'))_s - (z_t - z_s)|\\ 
    & \le |a_{r_l} (x_t  - x_{r_l} )|
    + |a(n')_{\tau_{l-1}} (x(n')_t - x(n')_{\tau_{l-1}} ) - a_{r_{l-1}} (x_{r_l} - x_{r_{l-1}})| \\
    & \quad + \sum_{k=l'+2}^{l-1} |a(n')_{\tau_{k-1}} (x(n')_{\tau_k} - x(n')_{\tau_{k-1}} ) - a_{r_{k-1}} (x_{r_k} - x_{r_{k-1}})| \\
    & \quad + |a(n')_{\tau_{l'}} (x(n')_{\tau_{l'+1}} - x(n')_s ) - a_{r_{l'}} (x_{r_{l'+1}} - x_s)| =: T_{31}+T_{32}+T_{33}+T_{34}.
\end{align*}
We can estimate $T_{33}$ and $T_{34}$ as the same argument $T_{12}$ and $T_{13}$ in (Case1). 
To estimate $T_{31}$,  let us note that $(I)$ implies 
\[
    r_l - \tilde{\delta}/2 < t \le \tau_l < r_l + \tilde{\delta}/2
\]
and so $|t-r_l| < \tilde{\delta}/2$.  Since  $|t-r_l|<|t-s|$, $(II)$ implies that
\begin{align*}
    T_{31} 
    \le \|a\|_{\infty} |x_t - x_{r_l}| 
    \le \|a\|_{\infty}\|x\|_{\alpha\text{-Hld}} |t - r_l|^{\alpha}
    \lesssim \eta^{\alpha-\beta} |t - s|^{\beta}.
\end{align*}
On the other hand, since $|t-r_l|<|t-s|$, $|t-\tau_{l-1}|<|t-s|$, and $|\tau_{l-1}-r_{l-1}|<|t-s|$,  $(I)$ to $(III)$ imply that
\begin{align*}
    T_{32} 
    & \le |a_{r_{l-1}}- a(n')_{\tau_{l-1}}| |x(n')_t - x(n')_{\tau_{l-1}}|  + |a_{r_{l-1}}| | x(n')_t - x(n')_{\tau_{l-1}} - (x_{r_l} -x_{r_{l-1}})|\\
    & \le |a_{r_{l-1}}- a(n')_{\tau_{l-1}}| \|x(n')\|_{\alpha \text{-Hld}} |t-\tau_{l-1}|^{\alpha} \\
    & \quad + \|a\|_{\infty} \Big \{ |x_t -x_{r_{l-1}} - (x(n')_t - x(n')_{r_{l-1}})| + | (x(n')_{r_{l-1}} - x(n')_{\tau_{l-1}})| +  |x_{r_l} -x_t|  \Big \} \\
    & \le \Big \{  |a_{r_{l-1}}- a(n')_{\tau_{l-1}}| \|x(n')\|_{\alpha \text{-Hld}}  + \|a\|_{\infty} \Big \{  \|x(n') -x\|_{\alpha\text{-Hld}}  \\
    & \quad + \|x(n') \|_{\alpha\text{-Hld}} | r_{l-1} - \tau_{l-1}|^{\alpha-\beta}   + \|x\|_{\alpha\text{-Hld}}  |t -r_l|^{\alpha -\beta}  \Big\} |t-s|^\beta\\
    & \lesssim (\eta + \eta^{\alpha-\beta})|t-s|^\beta.
\end{align*}

\item[(Case4)] $s\in[\tau_{l'+1},\tau_{l'+2}]$ and $t\in[\tau_{l},\tau_{l+1}]$.
\begin{align*}
    &|G_{\delta}(a(n'),x(n'))_t - G_{\delta}(a(n'),x(n'))_s - (z_t - z_s)|\\ 
    & \le |a(n')_{\tau_l} (x(n')_t - x(n')_{\tau_l} ) - a_{r_l} (x_t - x_{r_l})| \\
    & \quad + \sum_{k=l'+3}^l |a(n')_{\tau_{k-1}} (x(n')_{\tau_k} - x(n')_{\tau_{k-1}} ) - a_{r_{k-1}} (x_{r_k} - x_{r_{k-1}})| \\
    & \quad + |a(n')_{\tau_{l'+1}} (x(n')_{\tau_{l'+2}} - x(n')_s ) - a_{r_{l'+1}} (x_{r_{l'+2}} - x_{r_{l'+1}})| + |a_{\tau_{l'}} (x_{r_{l'+1}} - x_s )|\\
    & = :T_{41}+T_{42}+T_{43}+T_{44}.
\end{align*}
We can estimate $T_{41}$ and $T_{42}$ as the same argument $T_{11}$ and $T_{12}$ in (Case1). 
To estimate $T_{43}$ and $T_{44}$, let us note that by using the same argument in (Case2), $|s-r_{l'+1}|<\tilde{\delta}/2$.
Since $|r_{l'+2} -s|<|t-s|$, $|r_{l'+1} -s|<|t-s|$, and $|r_{l'+2} -\tau_{l'+2}|<|t-s|$, $(I)$ to $(III)$ imply that
\begin{align*}
    T_{43}
    & \le |a(n')_{\tau_{l'+1}} - a_{r_{l'+1}} | |x(n')_{\tau_{l'+2}} - x(n')_s | + |a_{r_{l'+1}}| |x(n')_{\tau_{l'+2}} - x(n')_s  - (x_{r_{l'+2}} - x_{r_{l'+1}})|\\
    & \le |a(n')_{\tau_{l'+1}} - a_{r_{l'+1}} | \|x(n')\|_{\alpha\text{-Hld}} |\tau_{l'+2} -s|^{\alpha} \\
    & \quad + \|a\|_{\infty} \Big \{  |x(n')_{r_{l'+2}} - x(n')_s  - (x_{r_{l'+2}} - x_s)|  + |x_{r_{l'+1}} -x_s| + |x(n')_{r_{l'+2}} - x(n')_{\tau_{l'+2}}| \Big \}\\
    & \le \Big \{  |a(n')_{\tau_{l'+1}} - a_{r_{l'+1}} | \|x(n')\|_{\alpha\text{-Hld}} + \|a\|_{\infty} \Big \{  \|x(n')-x\|_{\alpha\text{-Hld}} \\
    & \quad + \|x\|_{\alpha\text{-Hld}} |r_{l'+1} -s|^{\alpha-\beta}  + \|x (n')\|_{\alpha\text{-Hld}} |r_{l'+2} - \tau_{l'+2}|^{\alpha-\beta} \Big \} \Big\}|t-s|^\beta\\
    & \lesssim (\eta+\eta^{\alpha-\beta}) |t-s|^\beta.
\end{align*}
Since $|r_{l'+1} -s|<|t-s|$, 
\begin{align*}
    T_{44} \le \|a\|_{\infty} \|x\|_{\alpha\text{-Hld}} |r_{l'+1} -s|^{\alpha}
 \lesssim \eta^{\alpha-\beta}|t-s|^\beta.
\end{align*}
\item[(Case5)] $s\in[\tau_{l'+1},\tau_{l'+2}]$ and $t\in[\tau_{l+1},\tau_{l+2}]$.
\begin{align*}
    &|G_{\delta}(a(n'),x(n'))_t - G_{\delta}(a(n'),x(n'))_s - (z_t - z_s)|\\ 
    & \le |a(n')_{\tau_{l+1}}(x(n')_t  - x(n')_{\tau_{l+1}} )|  \\
    & \quad +|a(n')_{\tau_l} (x(n')_{\tau_{l+1}} - x(n')_{\tau_l} ) - a_{r_l} (x_t - x_{r_l})|\\
    & \quad + \sum_{k=l'+3}^l |a(n')_{\tau_{k-1}} (x(n')_{\tau_k} - x(n')_{\tau_{k-1}} ) - a_{r_{k-1}} (x_{r_k} - x_{r_{k-1}})| \\
    & \quad + |a(n')_{\tau_{l'+1}} (x(n')_{\tau_{l'+2}} - x(n')_s ) - a_{r_{l'+1}} (x_{r_{l'+2}} - x_{r_{l'+1}})|\\
    & \quad + |a_{\tau_{l'}} (x_{r_{l'+1}} - x_s )| = :T_{51}+T_{52}+T_{53}+T_{54}+T_{55}.
\end{align*}
We can estimate $T_{51}$ and $T_{52}$ as the same argument $T_{21}$ and $T_{22}$ in (Case2) and $T_{53}$ as the same argument $T_{12}$ in (Case1).
We can also estimate $T_{54}$ and $T_{55}$ as the same argument $T_{43}$ and $T_{44}$ in (Case4). 
\item[(Case6)] $s\in[\tau_{l'+1},\tau_{l'+2}]$ and $t\in[\tau_{l-1},\tau_{l}]$.
\begin{align*}
    &|G_{\delta}(a(n'),x(n'))_t - G_{\delta}(a(n'),x(n'))_s - (z_t - z_s)|\\ 
    & \le |a_{r_l} (x_t  - x_{r_l} )|
    + |a(n')_{\tau_{l-1}} (x(n')_t - x(n')_{\tau_{l-1}} ) - a_{r_{l-1}} (x_{r_l} - x_{r_{l-1}})| \\
    & \quad + \sum_{k=l'+3}^{l-1} |a(n')_{\tau_{k-1}} (x(n')_{\tau_k} - x(n')_{\tau_{k-1}} ) - a_{r_{k-1}} (x_{r_k} - x_{r_{k-1}})| \\
    & \quad + |a(n')_{\tau_{l'+1}} (x(n')_{\tau_{l'+2}} - x(n')_s ) - a_{r_{l'+1}} (x_{r_{l'+2}} - x_{r_{l'+1}})| + |a_{\tau_{l'}} (x_{r_{l'+1}} - x_s )|\\
    & = T_{61}+T_{62}+T_{63}+T_{64}+T_{65}.
\end{align*}
We can estimate $T_{61}$ and $T_{62}$ as the same argument $T_{31}$ and $T_{32}$ in (Case3) and $T_{63}$ as the same argument $T_{12}$ in (Case1).
We can also estimate $T_{64}$ and $T_{65}$ as the same argument $T_{43}$ and $T_{44}$ in (Case4). 
\item[(Case7)] $s\in[\tau_{l'-1},\tau_{l'}]$ and $t\in[\tau_{l},\tau_{l+1}]$.
\begin{align*}
    &|G_{\delta}(a(n'),x(n'))_t - G_{\delta}(a(n'),x(n'))_s - (z_t - z_s)|\\ 
    & \le |a(n')_{\tau_l} (x(n')_t - x(n')_{\tau_l} ) - a_{r_l} (x_t - x_{r_l})| \\
    & \quad + \sum_{k=l'+2}^l |a(n')_{\tau_{k-1}} (x(n')_{\tau_k} - x(n')_{\tau_{k-1}} ) - a_{r_{k-1}} (x_{r_k} - x_{r_{k-1}})| \\
    & \quad + |a(n')_{\tau_{l'}} (x(n')_{\tau_{l'+1}} - x(n')_{\tau_{l'}} ) - a_{r_{l'}} (x_{r_{l'+1}} - x_s)|\\
    & \quad + |a(n')_{\tau_{l'-1}} (x(n')_{\tau_{l'}} - x(n')_s )| =: T_{71}+T_{72}+T_{73}+T_{74}.
\end{align*}
We can estimate $T_{71}$ and $T_{72}$ as the same argument $T_{11}$ and $T_{12}$ in (Case1). 
To estimate $T_{73}$ and $T_{74}$, let us note that by using the same argument in (Case2), $|s-\tau_{l'}|<\tilde{\delta}/2$.
Since $|\tau_{l'+1}-\tau_{l'}|<|t-s|$, $| \tau_{l'+1}- r_{l'+1}|<|t-s|$ , and $| \tau_{l'} -s|<|t-s|$, $(I)$ to $(III)$ imply that 
\begin{align*}
    T_{73}
    & \le |a(n')_{\tau_{l'}} - a_{r_{l'}}| |x(n')_{\tau_{l'+1}} - x(n')_{\tau_{l'}}| \\
    & \quad + |a_{r_{l'}}| |x(n')_{\tau_{l'+1}} - x(n')_{\tau_{l'}} - (x_{r_{l'+1}} - x_s)|\\
    & \le |a(n')_{\tau_{l'}} - a_{r_{l'}}| \|x(n')\|_{\alpha\text{-Hld}} |\tau_{l'+1} - \tau_{l'}|^{\alpha} \\
    & \quad + \|a\|_{\infty} \Big \{  |x(n')_{\tau_{l'+1}} - x(n')_{\tau_{l'}} - (x_{\tau_{l'+1}} - x_{\tau_{l'}}) +|x_{\tau_{l'+1}} - x_{\tau_{l'}} - (x_{r_{l'+1}} - x_s)| \Big \}\\
    & \le \Big  \{ |a(n')_{\tau_{l'}} - a_{r_{l'}}| \|x(n')\|_{\alpha\text{-Hld}} + \|a\|_{\infty} \Big\{  \|x(n')-x\|_{\alpha\text{-Hld}} \\
    & \quad + \|x\|_{\alpha\text{-Hld}} ( |\tau_{l'+1} - r_{l'+1}|^{\alpha-\beta} + |s-\tau_{l'}|^{\alpha -\beta} ) \Big \} \Big \} |t-s|^\beta  \lesssim (\eta+\eta^{\alpha-\beta})|t-s|^\beta.
\end{align*}
Since $|\tau_{l'} -s|<|t-s|$,
\begin{align*}
    T_{74} \le \|a(n')\|_{\infty} \|x(n')\|_{\alpha\text{-Hld}} |\tau_{l'} -s|^{\alpha}
    \lesssim \eta^{\alpha-\beta} |t-s|^{\beta}.
\end{align*}
\item[(Case8)] $s\in[\tau_{l'-1},\tau_{l'}]$ and $t\in[\tau_{l+1},\tau_{l+2}]$.
\begin{align*}
    &|G_{\delta}(a(n'),x(n'))_t - G_{\delta}(a(n'),x(n'))_s - (z_t - z_s)|\\ 
    & \le |a(n')_{\tau_{l+1}}(x(n')_t  - x(n')_{\tau_{l+1}} )|  + |a(n')_{\tau_l} (x(n')_{\tau_{l+1}} - x(n')_{\tau_l} ) - a_{r_l} (x_t - x_{r_l})|\\
    & \quad + \sum_{k=l'+2}^l |a(n')_{\tau_{k-1}} (x(n')_{\tau_k} - x(n')_{\tau_{k-1}} ) - a_{r_{k-1}} (x_{r_k} - x_{r_{k-1}})| \\
    & \quad + |a(n')_{\tau_{l'}} (x(n')_{\tau_{l'+1}} - x(n')_{\tau_{l'}} ) - a_{r_{l'}} (x_{r_{l'+1}} - x_s)|+ |a(n')_{\tau_{l'-1}} (x(n')_{\tau_{l'}} - x(n')_s )|\\
    & = T_{81}+T_{82}+T_{83}+T_{84}+T_{85}.
\end{align*}
We can estimate $T_{81}$ and $T_{82}$ as the same argument $T_{21}$ and $T_{22}$ in (Case2) and $T_{83}$ as the same argument $T_{12}$ in (Case1).
We can also estimate $T_{84}$ and $T_{85}$ as the same argument $T_{73}$ and $T_{74}$ in (Case7). 
\item[(Case9)] $s\in[\tau_{l'-1},\tau_{l'}]$ and $t\in[\tau_{l-1},\tau_{l}]$.
\begin{align*}
    &|G_{\delta}(a(n'),x(n'))_t - G_{\delta}(a(n'),x(n'))_s - (z_t - z_s)|\\ 
    & \le |a_{r_l} (x_t  - x_{r_l} )|
    + |a(n')_{\tau_{l-1}} (x(n')_t - x(n')_{\tau_{l-1}} ) - a_{r_{l-1}} (x_{r_l} - x_{r_{l-1}})| \\
    & \quad + \sum_{k=l'+2}^{l-1} |a(n')_{\tau_{k-1}} (x(n')_{\tau_k} - x(n')_{\tau_{k-1}} ) - a_{r_{k-1}} (x_{r_k} - x_{r_{k-1}})| \\
    & \quad + |a(n')_{\tau_{l'}} (x(n')_{\tau_{l'+1}} - x(n')_{\tau_{l'}} ) - a_{r_{l'}} (x_{r_{l'+1}} - x_s)|+ |a(n')_{\tau_{l'-1}} (x(n')_{\tau_{l'}} - x(n')_s )|\\
    & = : T_{91}+T_{92}+T_{93}+T_{94}+T_{95}.
\end{align*}
We can estimate $T_{91}$ and $T_{92}$ as the same argument $T_{31}$ and $T_{32}$ in (Case3) and $T_{93}$ as the same argument $T_{12}$ in (Case1).
We can also estimate $T_{94}$ and $T_{95}$ as the same argument $T_{73}$ and $T_{74}$ in (Case7). 
\end{enumerate}
By using the all estimation of (Case1)-(Case9), we conclude that
\[
    |G_{\delta}(a(n'),x(n'))_t - G_{\delta}(a(n'),x(n'))_s - (z_t - z_s)| \lesssim (\eta+\eta^{\alpha-\beta})|t-s|^\beta,
\]
and this is the claim.
\end{proof}

\begin{definition}[Definition 4.2.14 in \cite{DeZe}]\label{3.1}
    Let $(\Omega,\mathcal{F},\mathsf{P})$ be a probability space and  $\delta, \epsilon > 0$. 
    Let  $Z^{\delta,\epsilon}$ and $Z^{\epsilon}$ be random functions on a metric space $(E,d_E)$ respectively. 
    We say that $\{Z^{\delta,\epsilon}\}_{\delta,\epsilon > 0}$ are exponentially good approximation of $\{ Z^{\epsilon} \}_{\epsilon > 0}$ if for every $\eta>0$,
    $$\{\omega  \in \Omega: d_E (Z^{\delta ,\epsilon} (\omega),Z^{\epsilon} (\omega) ) > \eta \} \in \mathcal{F},$$
    and   
    \begin{align*}
        \lim_{\delta \searrow 0}\limsup_{\epsilon \searrow 0} \epsilon \log \PROB{d_E (Z^{\delta,\epsilon} ,Z^{\epsilon}) > \eta } = - \infty.
    \end{align*}
\end{definition}

Now let us note that one can represent $G_\delta$ as a stochastic integral. For $ \{ \tau_{k}^\delta \} = \{\tau_{k}^\delta (A^\epsilon)  \} $, let 
    \[
    \Psi_{\delta} (A^\epsilon)_t 
    := \sum_{k=1}^{\infty} A^\epsilon_{\tau^{\delta}_{k-1} } 1_{ (\tau^{\delta}_{k-1}, \tau^{\delta}_k ]} (t), \quad t \in [0,1].
    \]
Then the definition of $\{\tau_k^\delta\}$, the process $\{ \Psi_{\delta} (A^\epsilon) \}_{\delta,\epsilon}$ is a family of adapted, left continuous with right limits processes on $[0,1]$. 
Therefore, we have that
\begin{align}\label{inte}
    G_{\delta} (A^\epsilon,X^\epsilon)_t = \int_0^t \Psi_{\delta} (A^\epsilon)_r  \D X^\epsilon_r  , \quad t \in [0,1],
\end{align}  
where the integration in the right hand side is It\^{o} integral. 
By Remark \ref{welldef},  we have $\Psi_\delta( A^{\epsilon})\cdot X^{\epsilon} \in C^{\beta\text{-Hld}}_0([0,1])$.

\begin{lemma}\label{ega}
If $\{X^{\epsilon} \}_{\epsilon}$ is $\alpha$-Uniformly Exponentially Tight,  $\{ \Psi_{\delta} ( A^{\epsilon}) \cdot X^{\epsilon} \}_{\epsilon,\delta}$ is exponentially good approximation of $\{ A^{\epsilon} \cdot X^{\epsilon} \}_{\epsilon}$ on $ C_0^{\beta \text{-Hld}}([0,1])$: for all $\eta>0$ and $M>0$, there exists $\delta(\eta,M)>0$ such that if $0 < \delta < \delta(\eta,M)$, 
 \[
     \limsup_{\epsilon \searrow 0} \epsilon \log  \PROB { \|  \left(  \Psi_{\delta} ( A^{\epsilon}) -  A^{\epsilon} \right) \cdot X^{\epsilon} \|_{\beta \text{-Hld}}> \eta } \le -M.
 \]
\end{lemma}

\begin{proof}
 The measurability requirement in Definition \ref{3.1} is satisfied by the fact that  $A^{\epsilon}\cdot X^{\epsilon}$ is an adapted process and 
 \[
 \sup_{0\le s < t \le 1} \frac{|x_t-x_s|}{|t-s|^\alpha} = \sup_{0\le s < t \le 1, s,t \in \bbQ} \frac{|x_t-x_s|}{|t-s|^\alpha}, \quad x \in C^\alpha ([0,1]).
  \]

To verify the remaining assertion, fix $M>0$ and $\eta>0$. Take $ K_M >0$ such that (\ref{1.1}) holds. 
For this $\eta$ and $K_M$, taking $\delta$ small enough ($\eta \delta^{-1} >K_M$). 
Note that  the definition of $\{ \tau_k^\delta(A^\epsilon) \}$ implies that $| A^\epsilon_t - \Psi_\delta(A^\epsilon)_t| \le \delta$, and so $ \delta^{-1} (A^\epsilon - \Psi_\delta(A^\epsilon) ) \in \mathcal{B} ([0,1],\bbR)$. 
Then one has that
\begin{align*}
\PROB{ \| (A^{\epsilon} - \Psi_{\delta} (A^\epsilon) ) \cdot X^{\epsilon} \|_{\beta \text{-Hld}} > \eta } 
& = \PROB{  \left| \left| \delta^{-1} (A^{\epsilon} -  \Psi_{\delta} (A^\epsilon)) \cdot X^{\epsilon} \right| \right|_{\beta \text{-Hld}} > \eta  \delta^{-1}} \\
& \le \PROB{ \left| \left| \delta^{-1} (A^{\epsilon} - \Psi_{\delta} (A^\epsilon)) \cdot X^{\epsilon} \right| \right|_{\beta \text{-Hld}} > K_M  }\\
& \le \sup_{U \in \mathcal{B} ([0,1],\bbR)} \PROB{  \|U \cdot X^{\epsilon} \|_{\beta\text{-Hld}}  > K_M},
\end{align*}
and so (\ref{1.1}) implies the claim.
\end{proof}

 \begin{definition}
     Let $X^\epsilon$ is a random function taking value a Banach space $E$. We say that $\{X^{\epsilon} \}_{\epsilon > 0}$ is exponentially tight if for all $M>0$, there exists a compact set $K_M$ on $E$ such that
     \begin{align*}
         \limsup_{\epsilon \searrow 0 } \epsilon \log  \PROB{  X^{\epsilon} \in K^c_M} \le -M.
     \end{align*}
 \end{definition}

\begin{proof}[Proof of Theorem \ref{main1}]
We will first prove that $\{ (A^{\epsilon},X^{\epsilon}, A^{\epsilon} \cdot X^{\epsilon} ) \}_{\epsilon}$ is exponentially tight on $C([0,1])\times C_0^{\alpha \text{-Hld}}([0,1]) \times C_0^{\beta \text{-Hld}}([0,1])$.
Since $C([0,1])\times C_0^{\alpha\text{-Hld}}([0,1])$ is Polish space, the assumption implies that $\{ (A^{\epsilon},X^{\epsilon})\}_{\epsilon}$ is exponentially tight on $C([0,1]) \times C_0^{\alpha \text{-Hld}}([0,1])$ (see Exercise 4.1.10 in \cite{DeZe}). 
Lemma \ref{G} implies that $(a,x) \mapsto ,G_{\delta} (a,x)$ is almost compact from $C ([0,1])\times C_0^{\alpha \text{-Hld}}([0,1]) $ into $C_0^{\beta \text{-Hld}}([0,1])$. 
Since $\{ X^{\epsilon} \}_{\epsilon}$ is $\alpha$-uniformly exponentially tight, \eqref{inte} and Lemma \ref{ega} imply that $\{ G_{\delta} (A^{\epsilon},X^{\epsilon}) \}_{\delta, \epsilon}$ is exponentially good approximation of $\{ A^{\epsilon} \cdot X^{\epsilon}  \}_{\epsilon}$ on $  C_0^{\beta \text{-Hld}}([0,1])$.
Therefore Theorem 7.1 in \cite{Ga08} implies that $\{ (A^{\epsilon},X^{\epsilon}, A^{\epsilon} \cdot X^{\epsilon} ) \}_{\epsilon}$ is exponentially tight on $C([0,1])\times C_0^{\alpha \text{-Hld}}([0,1]) \times C_0^{\beta \text{-Hld}}([0,1])$. 

Let $C([0,\infty))$ is the set of all continuous function with the metric 
\[
d_{\infty} (x,y) := \sum_{n=1}^\infty \frac{1}{2^n} ( 1 \wedge \sup_{t \in [0,n]} |x_t-y_t| ), \quad x,y \in C([0,\infty)), 
\]
and let $(D([0,\infty)), d_\infty)$ is the set of all cadlag function equipped with $d_\infty$.
 Let $F_1 : C([0,1]) \rightarrow C([0,\infty))$ as $ F_1(x)_t := x_{t \wedge 1}$ and let $F_2 : C([0,1]) \rightarrow (D([0,\infty)) ,d_{\infty})$ as $ F_2(x)_t := x_t 1_{ [0,1)}(t)$. Since $F_1$ and $F_2$ are continuous and injective respectively, the contraction principle implies that $\{(F_2(A^{\epsilon}),F_1(X^{\epsilon}) )\}$ satisfies the LDP on $( D([0,\infty)),d_\infty) \times C([0,\infty))$ with good rate function 
\begin{align*}
    I^{(1)} (\tilde{a},\tilde{x}) :=
    \begin{cases}
      I^{\#}(a,x), & \exists (a,x) \in C([0,1]) \times C([0,1]) \text{ s.t. } (\tilde{a}, \tilde{x} ) = (F_2(a),F_1(x)) ,\\
      \infty, & \text{otherwise}.
     \end{cases}
\end{align*}
Note that for a real valued adapted left continuous with right limits  process $H$ and a real valued semi-martingale $V$, we have that 
\begin{align}\label{stop}
H \cdot F_1(V) = F_2(H) \cdot F_1(V) = F_1 ( H \cdot V),
\end{align}
see Theorem 5.6 in \cite{LeG}, for example. Then we have that for all $t \in [0,\infty)$, $U \in \mathcal{S}$, and $K >0$,
\begin{align*}
    \PROB{\sup_{s \le t}|(U_{-}\cdot F_1(X^{\epsilon}))_s| > K } 
    & = \PROB{\sup_{s \le t} |F_1 (U_{-} \cdot X^{\epsilon})_s| > K }\\
    & \le \PROB{\sup_{s \le 1} |F_1 (U_{-} \cdot X^{\epsilon})_s| > K }\\
    & \le \PROB{ \| ( U_{-}\lvert_{[0,1]}) \cdot X^{\epsilon}\|_{\alpha \text{-Hld}} > K }\\
    & \le \sup_{\tilde{U} \in \mathcal{B}([0,1],\bbR)} \PROB{ \| \tilde{U} \cdot X^{\epsilon}\|_{\alpha \text{-Hld}} > K },
\end{align*}
where $U_{-}\lvert_{[0,1]}$ is the restriction of $U_{-}$ to $[0,1]$. Because $\{ X^{\epsilon} \}$ is the $\alpha$-Uniformly Exponentially tight, $\{F_1(X^{\epsilon})\}$ is a martingale  satisfying \eqref{UET}.  
Then Lemma \ref{Ga} implies that 
\[
\{ (F_2(A^{\epsilon} ) , F_1( X^{\epsilon} ), F_2(A^{\epsilon}) \cdot F_1(X^{\epsilon}) ) \}_{\epsilon} =\{ (F_2(A^{\epsilon} ) , F_1( X^{\epsilon} ), F_1 (A^{\epsilon}  \cdot X^{\epsilon}) ) \}_{\epsilon}
\]
 satisfies the LDP on $ ( D([0,\infty)) ,d_\infty) \times C([0,\infty)) \times C([0,\infty))$ with good rate function
\begin{align*}
I^{(2)}(a,x,\tilde{z}) 
& = \begin{cases} 
I^{\#}(a,x) , &  \tilde{z} = F_2(a) \cdot F_1(x), \ x \in \textbf{BV}  \\
\infty, & \text{otherwise}
\end{cases}
\\& = \begin{cases} 
I^{\#}(a,x) , &  \tilde{z} = F_1(a  \cdot x), \ x \in \textbf{BV}  \\
\infty, & \text{otherwise}.
\end{cases}
\end{align*}
By using Lemma 4.1.5 $(b)$ in \cite{DeZe}, it is straightforward to prove that $\{ (F_2(A^{\epsilon} ) , F_1( X^{\epsilon} ), F_1 (A^{\epsilon}  \cdot X^{\epsilon}) ) \}_{\epsilon}$  satisfies the LDP on $  ( F_2(C([0,1])) ,d_\infty) \times ( F_1(C([0,1]) ) ,d_\infty)  \times ( F_1(C([0,1])) ,d_\infty) $ with good rate function $I^{(2)}$.
 
  Let $\mathcal{E} := ( F_2(C([0,1])) ,d_\infty) \times ( F_1(C([0,1]) ) ,d_\infty)\times ( F_1(C([0,1]) ) ,d_\infty)$ and let $F_3 : \mathcal{E} \to  C([0,1]) \times C([0,1]) \times C([0,1])$ as 
  \[
  F_3 (\tilde{a},\tilde{x},\tilde{z})_t 
  := \begin{cases}
  (\tilde{a}_t,\tilde{x}_t,\tilde{z}_t) & t \in [0,1)\\
   (\lim_{t \nearrow 1}\tilde{a}_t,\tilde{x}_1,\tilde{z}_1) & t =1.
  \end{cases}
 \]  
 Because $F_3$ is continuous and injective, the contraction principle implies that $\{ (A^{\epsilon} ,X^{\epsilon}, A^{\epsilon} \cdot X^{\epsilon} ) \}_{\epsilon}$ satisfies the LDP on $C([0,1]) \times C([0,1]) \times C([0,1])$ with good rate function
\begin{align*}
I(a,x,z) = 
\begin{cases}
I^{\#}(a,x) & z = a \cdot x, \ x \in \textbf{BV} ,\\
\infty, & \text{otherwise}.
\end{cases}
\end{align*}
Therefore the inverse contraction principle (Theorem 4.2.4 in \cite{DeZe}) implies that 
$$\{(A^{\epsilon},X^{\epsilon}, A^{\epsilon} \cdot X^{\epsilon} ) \}_{\epsilon}$$
 satisfies the LDP on $C([0,1]) \times C_0^{\alpha\text{-Hld}}([0,1]) \times C_0^{\beta \text{-Hld}}([0,1])$ with good rate function $I$, and this is the claim.
\end{proof}

\begin{proof}[Proof of Proposition \ref{main2}]
To verify $A\cdot B^{\epsilon} \in C^{\alpha\text{-Hld}}([0,1],\bbR)$, we fix any  $(\mathcal{F}_t)$-adapted continuous processes $A$ on $[0,1]$ and $\alpha \in [1/3,1/2)$.
  For brevity, we assume that $\epsilon =1$. 
  Let $F_2 : C([0,\infty)) \to D([0,\infty))$ as $F_2(x)_t := x_t 1_{[0,1)} (t) + 1_{[1,\infty)} (t)$ and  $\tau_n := \inf\{t \ge 0 : |F_2(A_t)| > n \}$. 
Then we have that $\tau_n$ is a $(\mathcal{F}_t)$-stopping time,  $\tau_n \le \tau_{n+1}$ a.s. for all $n \in \bbN$, and $\tau_n \to \infty $ as $n \to \infty $ a.s.
One can also prove that $\sup_{t \in [0,1 \wedge \tau_n ]} |A_t| \le n$.
Let $\|x\|_{\alpha\text{-Hld},[0,c]} := |x_0|+ \sup_{0\le s<t\le c} \frac{|x_t-x_s|}{|t-s|^\alpha}$.
Then we have that
\begin{align*}
    \PROB{\|A\cdot B\|_{\alpha \text{-Hld},[0,1]} < \infty} 
     &\ge   \PROB{ \cap_{n=1}^\infty \{ \|A\cdot B\|_{\alpha \text{-Hld},[0,1\wedge \tau_n]}  < \infty \} } \\
    & =  \lim_{n \to \infty }   \PROB{  \|A\cdot B\|_{\alpha \text{-Hld},[0,1\wedge \tau_n]}  < \infty  },
\end{align*}
and so it is sufficient to prove that for all $n \in \bbN$,
\begin{align}\label{ab}
  \PROB{  \|A\cdot B\|_{\alpha \text{-Hld},[0,1\wedge \tau_n]}  < \infty  }=1.
\end{align}
Let $\tilde{A}^{(n)}_t := A_t1_{[0,1\wedge \tau_n]} (t)+ 1_{(1\wedge \tau_n, \infty)} (t)$, then  $\|A\cdot B \|_{\alpha,[0,1\wedge \tau_n] } = \|\tilde{A}^{(n)}\cdot B \|_{\alpha,[0,1\wedge \tau_n] } $.
Since 
\[
\braket{\tilde{A}^{(n)} \cdot B}_t = \int_0^t (\tilde{A}^{(n)} )^2_r \D r \to \infty, \quad t \to \infty,
\]
and $\tilde{A}^{(n)} \cdot B$ is a local continuous martingale, the Dambis-Dubins-Schwarz's Theorem implies that there exists a Brownian motion $\tilde{B}$ such that
\begin{equation}\label{changea}
(\tilde{A}^{(n)} \cdot B)_t = \tilde{B}_{\braket{\tilde{A}^{(n)} \cdot B}_t }, \quad t \in [0,\infty).
\end{equation}
We now restrict $\tilde{A}^{(n)} \cdot B$ on $[0,1]$.
Since  $\sup_{t \in [0,1\wedge \tau_n] } |\tilde{A}^{(n)}_t| = \sup_{t \in [0,1\wedge \tau_n] } |{A}_t|  \le n$, one has that $ \braket{\tilde{A}^{(n)} \cdot B}_t \le n^2 $ for $t\in[0,1]$, and this implies that for $0 \le s<t \le 1$,
\begin{align*}
\left |\tilde{B}_ {\braket{\tilde{A}^{(n)} \cdot B}_t }  - \tilde{B}_{ \braket{\tilde{A}^{(n)} \cdot B}_s } \right | 
& \le \|\tilde{B}\|_{\alpha \text{-Hld},[0,n^2]} \left | \braket{\tilde{A}^{(n)} \cdot B}_t  -  \braket{\tilde{A}^{(n)} \cdot B}_s  \right|^{\alpha}\\
& \le n^{2\alpha} \|\tilde{B}\|_{\alpha \text{-Hld},[0,n^2]} \left |  t-s  \right|^{\alpha}.
\end{align*}
Combined with (\ref{changea}), we have 
\[
\|A\cdot B\|_{\alpha \text{-Hld},[0,1\wedge\tau_n]} = \|\tilde{A}^{(n)} \cdot B\|_{\alpha \text{-Hld},[0,1]} \le n^{2\alpha} \|\tilde{B}\|_{\alpha \text{-Hld},[0,n^2]}, \quad \text{a.s.}
\]  
and this implies that for all $n$, we have \eqref{ab}. Hence we conclude that $A\cdot B^{\epsilon} \in C^{\alpha\text{-Hld}}([0,1],\bbR)$,

To verify the $\alpha$-Uniformly Exponentially Tightness of $\{ B^\epsilon \}$,  we fix $M>0$ and  $U \in \mathcal{B}([0,1],\bbR)$. To regard $U$ as a process on $[0,\infty)$, define $U_t = 1 $, $t \in (1,\infty)$. Then we have 
\[
\braket{U \cdot B}_t = \int_0^t U^2_r \D r \to \infty, \quad t \to \infty.
\]
Since $U \cdot B$ is a continuous martingale, the Dambis-Dubins-Schwarz's Theorem implies that there exists a Brownian motion $\bar{B}'$ such that
\begin{equation}\label{change}
(U \cdot B)_t = \bar{B}'_{\braket{U\cdot B}_t }, \quad t \in [0,\infty).
\end{equation}
We now restrict $U\cdot B$ on $[0,1]$. Since  $\sup_{t \in [0,1] } |U_t| \le 1$, one has that $\braket{U \cdot B}_t \le t$, and this implies that for $0 \le s<t \le 1$,
\begin{align*}
\left |\bar{B}'_ {\braket{U\cdot B}_t }  - \bar{B}'_{ \braket{U\cdot B}_s } \right | 
& \le \|\bar{B}'\|_{\alpha \text{-Hld}} \left | \braket{U \cdot B}_t  -  \braket{U \cdot B}_s  \right|^{\alpha}\\
& \le \|\bar{B}'\|_{\alpha \text{-Hld}} \left |  t-s  \right|^{\alpha}.
\end{align*}
Combined with (\ref{change}), we have 
\[
\|U \cdot B\|_{\alpha \text{-Hld}} \le \|\bar{B}'\|_{\alpha \text{-Hld}}.
\] 
Since $\tilde{B}$ is an one dimensional Brownian motion, $\|\tilde{B}\|_{\alpha\text{-Hld}}$ has a Gaussian tail (Corollary 13.14 in \cite{FV10}).
Therefore, there exists $c>0$ such that for all $K>0$, 
\begin{align*}
\PROB{  \left| \left| U\cdot B^{\epsilon} \right| \right|_{\alpha \text{-Hld}} >  K} 
& = \PROB{ \left| \left| U \cdot B \right| \right|_{\alpha \text{-Hld}} > \epsilon^{-1/2} K  } \\
& \le \PROB{ \|\bar{B}'\|_{\alpha\text{-Hld}} > \epsilon^{-1/2} K } \le c^{-1} \exp{(-c\epsilon^{-1}K^2)}.
\end{align*}
This implies that
\[
\limsup_{\epsilon\searrow 0} \epsilon \log \sup_{U \in \mathcal{B}([0,1],\bbR) }
\PROB{ \left| \left| U \cdot B^{\epsilon} \right| \right|_{\alpha \text{-Hld}} > K}  \le -cK^2,
\]
and so take $K_M$ large enough ($c K^2_M >M$), then we conclude that 
\[
\limsup_{\epsilon \searrow 0} \epsilon \log \sup_{U \in \mathcal{B}([0,1],\bbR) } 
\PROB{ \left| \left| U \cdot B^{\epsilon} \right| \right|_{\alpha \text{-Hld}} > K_M }  \le -M,
\]
and this is the claim.

It remains to verify $(ii)$.
The proof of  $\bar{A} \cdot \bar{B}^\epsilon \in C^{\alpha\text{-Hld}}$  follows from a simple modification of $(i)$ and so we will focus on  $\alpha$-Uniformly Exponentially Tightness of $\{\bar{B}^\epsilon \}$.
   We fix $U \in \mathcal{B} ([0,1] ,\bbR)$ (note that $U$ is an $(\mathcal{F}^\epsilon)$-adapted process). 
Let $\tilde{U}_t := U_t 1_{[0,1]}(t) + 1_{(1,\infty) } (t)$. Then we have that for all $\epsilon >0$,
\[
\braket{\tilde{U} \cdot \bar{B}^\epsilon}_t = \int_0^t (\tilde{U} )^2_r \D \braket{\bar{B}^\epsilon}_r =  \int_0^t (\tilde{U} )^2_r \D \epsilon r   \to \infty, \quad t \to \infty,
\]
and so for each $\epsilon >0$, there exists a Brownian motion $\tilde{B}^{(\epsilon)} $ such that 
\[
(\tilde{U} \cdot \bar{B}^\epsilon)_t = \tilde{B}^{(\epsilon)}_{\braket{\tilde{U} \cdot \bar{B}^\epsilon}_t} =  \tilde{B}^{(\epsilon)}_{ \epsilon \left(  \int_0^t (\tilde{U} )^2_r \D r \right)}, \quad t \in [0,\infty).
\]
Since $\sup_{t \in [0,1]} | \int_0^t (\tilde{U} )^2_r \D r | \le  1 $, we have that 
\begin{align*}
\left|  \tilde{B}^{(\epsilon)}_{ \epsilon \left(  \int_0^t (\tilde{U} )^2_r \D r \right)} -  \tilde{B}^{(\epsilon)}_{ \epsilon \left(  \int_0^s (\tilde{U} )^2_r \D r \right)} \right|
& \le \| \tilde{B}^{(\epsilon)}_{\epsilon \cdot}  \|_{\alpha \text{-Hld},[0,1]} \left |  \int_0^t (\tilde{U} )^2_r \D r -  \int_0^s (\tilde{U} )^2_r \D r \right|^\alpha\\
& \le \| \tilde{B}^{(\epsilon)}_{\epsilon \cdot}  \|_{\alpha \text{-Hld} ,[0,1]} \left | t-s  \right|^\alpha,
\end{align*}
and so one has 
\[
\| U \cdot \bar{B}^\epsilon \|_{\alpha ,[0,1]} \le \| \tilde{B}^{(\epsilon)}_{\epsilon \cdot}  \|_{\alpha \text{-Hld} ,[0,1]}, \quad \text{a.s.}
\]
Since $ \tilde{B}^{(\epsilon)}$ is a Brownian motion, one can prove that 
\[
\EXP{|\tilde{B}^{(\epsilon)}_{\epsilon t} - \tilde{B}^{(\epsilon)}_{\epsilon s}  |^p}^{1/p} \le \sqrt{\epsilon} \sqrt{p} |t-s|^{1/2}.
\]
Then the argument of Gaussian tails in Lemma A.17 in \cite{FV10}, one can prove that there exists $c>0$ ($\epsilon$-uniform) such that 
\[
\PROB{ \| \tilde{B}^{(\epsilon)}_{\epsilon \cdot}  \|_{\alpha \text{-Hld} ,[0,1]} \ge K } \le c \exp{\left ( -\frac{K^2}{8e c_{\alpha} \epsilon} \right)},
\]
and so 
\begin{align*}
\PROB{  \left| \left| U\cdot \bar{B}^{\epsilon} \right| \right|_{\alpha \text{-Hld}, [0,1]} >  K} 
& \le \PROB{ \|\tilde{B}^{(\epsilon)}_{\epsilon \cdot} \|_{\alpha\text{-Hld},[0,1]} > \ K } \le c \exp{\left ( -\frac{K^2}{8e c_{\alpha} \epsilon} \right)},
\end{align*}
and  we have the claim. 
\end{proof}

\subsection{Proof of Theorem \ref{main4} and \ref{main5}}\label{sec32}
\begin{proof}[Proof of Theorem \ref{main4}]
We fix $\alpha \in [1/3,1/2)$ and $1/2 > \alpha' > \alpha $ such that $\{(A^{\epsilon}, \tilde{A}^{\epsilon},X^{\epsilon} ) \}_{\epsilon }$ satisfies the LDP on $C([0,1]) \times C([0,1]) \times C_0^{\alpha'\text{-Hld}}([0,1])$ with good rate function $J^{\#}$. 
Take $\alpha'' $ with $\alpha ' > \alpha'' > \alpha$. 
Note that $A^\epsilon \cdot X^\epsilon \in C^{\alpha'\text{-Hld}}([0,1])$ and $\{ X^{\epsilon} \}_{\epsilon }$ is $\alpha'$-Uniformly Exponentially Tight by Proposition \ref{main2} $(i)$. 
Then Theorem \ref{main1} implies that 
$\{( A^{\epsilon}, \tilde{A}^{\epsilon}, X^{\epsilon}, A^{\epsilon} \cdot X^{\epsilon}  ) \}_{\epsilon}$ satisfies the LDP on $ C([0,1]) \times C([0,1]) \times C_0^{\alpha'\text{-Hld}}([0,1]) \times C_0^{\alpha'' \text{-Hld}}([0,1])$ with good rate function 
\[
    J^{(1)} (a,\tilde{a}, x, z) 
    := \begin{cases}
      J^{\#}(a,\tilde{a},x) &  z =   a \cdot x, \ x \in \textbf{BV}\\
      \infty, & \text{otherwise}.
      \end{cases}
\]
Since $x \mapsto \int_0^{\cdot} x_r \D r $ is continuous from $C([0,1])$ to $C^{1\text{-Hld}}([0,1])$, the contraction principle implies that $\{ Z^{\epsilon} := (A^{\epsilon} \cdot X^{\epsilon}, \tilde{A}^{\epsilon} \cdot \Lambda)  \}_{\epsilon}$ satisfies the LDP on $C_0^{\alpha'' \text{-Hld}} ([0,1])\times C^{1 \text{-Hld}} ([0,1])$ with good rate function 
\[
    J^{(2)} (z^{(1)},z^{(2)}) 
    := \inf \left\{ J^{\#}(a,\tilde{a},x) :  (z^{(1)},z^{(2)}) = ( a \cdot x, \tilde{a} \cdot \Lambda ), \ x \in \textbf{BV}  \right\}.
\]
We define $F: C_0^{\alpha'' \text{-Hld}} ([0,1], \mathbb{R}) \times C^{1\text{-Hld}} ([0,1],\mathbb{R}) \rightarrow G \Omega^{\alpha\text{-Hld}}([0,1], \mathbb{R}^2)$ as \eqref{defF}:
\[
    F(z)_{st} : =  (1,z_{st}, \mathbf{z}_{st}),\quad z \in C_0^{\alpha ''\text{-Hld}} (\mathbb{R}) \times C^{1\text{-Hld}} (\mathbb{R}).
\] 
We first prove $F(z) \in G\Omega^{\alpha\text{-Hld}}([0,1],\mathbb{R}^2)$. 
It is straightforward to show that $F(z) = (1,z,\mathbf{z})$ has the Chen's relation: for $s\le u \le  t$,
\begin{align*}
    z_{st} = z_{su} + z_{ut}, \quad \mathbf{z}_{st} = \mathbf{z}_{su} + \mathbf{z}_{ut} + z_{su} \otimes z_{ut}.
\end{align*}
We also have that 
\begin{align*}
    \sup_{0\le s < t \le 1} \frac{|z_{st}|}{|t-s|^{\alpha}} < \infty, \quad \sup_{0\le s < t \le 1} \frac{|\mathbf{z}_{st}|}{|t-s|^{2 \alpha}} < \infty,
\end{align*}
by the estimate of Young integral (Theorem 6.8 in \cite{FV10})
\begin{align}\label{Young}
        \left |  \int_s^t ( z_r^{(i)} -z_s^{(i)} )  \D z_r^{(j)}  \right| 
    \lesssim \|z^{(i)}\|_{\alpha^{(i)} \text{-Hld}} \|z^{(j)}\|_{\alpha^{(j)} \text{-Hld}} |t-s|^{2\alpha},
\end{align}
where $\alpha^{(i)} = \alpha''$ if $i=1$, otherwise $\alpha^{(i)} = 1$.
By using Theorem 5.25 in \cite{FV10} and the estimate of Young integral, one can prove that $F(z) \in G\Omega^{\alpha\text{-Hld}}([0,1],\mathbb{R}^2)$. 
Now we will show that $F$ is continuous. 
Assume that $z(n) \to z$ in $C_0^{\alpha'' \text{-Hld}} ([0,1],\mathbb{R}) \times C^{1\text{-Hld}} ([0,1],\mathbb{R})$. 
It is sufficient to consider the continuity of $\mathbf{z}^{(ij)}$. It is obvious when $i=j=1$. In the other case, by using (\ref{Young}), we have that
 \begin{align*}
      |\mathbf{z}(n)^{(ij)}_{st} - \mathbf{z}^{(ij)}_{st}|
     & \le \left|\int_s^t \{z(n)_r^{(i)} - z(n)_s^{(i)} - z^{(i)}_r + z^{(i)}_s  \}  \D z(n)_r^{(j)} \right| \\
     & \quad + \left|\int_s^t \{ z^{(i)}_r - z^{(i)}_s  \}  \D \{ z(n)_r^{(j)} - z^{(j)}_r \}  \right| \\
     & \lesssim \|z(n)^{(i)}-z^{(i)}\|_{\alpha^{(i)} \text{-Hld}} \|z(n)^{(j)}\|_{\alpha^{(j)} \text{-Hld}} |t-s|^{2\alpha} \\
     & \quad + \|z^{(i)}\|_{\alpha^{(i)} \text{-Hld}} \|z(n)^{(j)} - z^{(j)}\|_{\alpha^{(j)} \text{-Hld}} |t-s|^{2\alpha},
 \end{align*}
 and so we have that $F$ is continuous.
 Hence the contraction principle implies that $\{\mathbb{Z}^{\epsilon} = F(Z^{\epsilon}) \}_{\epsilon }$ satisfies the LDP on $G\Omega^{\alpha\text{-Hld}}([0,1],\mathbb{R}^2)$ with good rate function
 \[
    J^{(3)} (\tilde{\mathbf{z}}) 
    := \inf \left\{ J^{\#}(a,\tilde{a},x) :  \tilde{\mathbf{z}} = F(a \cdot x, \tilde{a} \cdot \Lambda) , \ x \in \textbf{BV} \right\},
\]
Because the solution map $\Phi: G\Omega^{\alpha\text{-Hld}}([0,1],\mathbb{R}^2) \rightarrow C^{\alpha \text{-Hld}} ([0,1]) $ is continuous, the contraction principle implies that $\{ Y^{\epsilon} = \Phi \circ F(Z^{\epsilon})  \}_{\epsilon}$ satisfies the LDP on $C^{\alpha\text{-Hld}}([0,1])$ with good rate function
\[
J (y) := 
\inf \left\{ J^{\#}(a,\tilde{a},x) :  y = \Phi \circ F(a \cdot x , \tilde{a} \cdot \Lambda ), \ x \in \textbf{BV}   \right\},
\]
and so Theorem \ref{thm:33} implies the claim.
\end{proof}



\begin{lemma}\label{BET}
We fix $\alpha \in (1/3,1/2)$, $\gamma \in(0,1)$, and let $X$ and $\hat{X}$ be a stochastic process defined as \eqref{defX} respectively.
Then  $\left \{ \left ( f(\hat{X}^{\epsilon}, \cdot  ),f^2(\hat{X}^{\epsilon}, \cdot  ), X^{\epsilon} \right) \right\}_\epsilon$ satisfies the LDP on $C([0,1]) \times C([0,1])\times C_0^{\alpha \text{-Hld}}([0,1])$ with good rate function
\begin{align*}
    \tilde{J}^{\#} (a,\tilde{a},x)  = \inf \left \{ \frac{1}{2} \| (w,w^\perp)\|^2_{\mathcal{H}} :   (a,\tilde{a})  = F_f \circ  \mathbb{K}(w,w^\perp) , \ (w,w^\perp) \in \mathcal{H} \right \}.
\end{align*}
\end{lemma}

\begin{proof}
It is well-known that $\{ \epsilon^{1/2} (W,W^{\perp})\}_{\epsilon}$ satisfies the LDP on $C([0,1]) \times C([0,1])$ with good rate function 
\begin{align*}
I^{(0)} (w,w^\perp) := \begin{cases}
\frac{1}{2} \| (w,w^\perp) \|_{\mathcal{H}}^2, & (w,w^\perp ) \in \mathcal{H},\\
\infty, & \text{otherwise}.
\end{cases}
\end{align*}
Since $\alpha \in (0,1/2)$ and $\|(W,W^{\perp})\|_{\alpha \text{-Hld}}$ has a Gaussian tails, the inverse contraction principle (see Theorem~4.2.4 in \cite{DeZe}) implies that $\{ \epsilon^{1/2} (W,W^{\perp})\}_{\epsilon}$ satisfies the LDP on $C_0^{\alpha \text{-Hld}}([0,1]) \times C_0^{\alpha \text{-Hld}}([0,1])$ with good rate function $ I^{(0)}$ (here we use the argument of Proposition 13.43 in \cite{FV10}). 
By Theorem~1 in \cite{Fu23}, the map $f \mapsto \mathcal{K} f$ is continuous from $C^{\alpha \text{-Hld}}([0,1])$ to $C^{\gamma \text{-Hld}}([0,1])$. 
Then the contraction principle implies that $\left\{ \epsilon^{1/2} (\hat{X},X) = \epsilon^{1/2} (\mathcal{K}W, \rho W + \sqrt{1-\rho^2}W^{\perp} ) \right \}_{\epsilon}$ satisfies the LDP on $C([0,1]) \times C^{\alpha \text{-Hld}}_0([0,1])$ with good rate function 
\[
\tilde{J}^{(1)} (\hat{x},x)  = \inf \left \{ \frac{1}{2} \| (w,w^\perp)\|_{\mathcal{H}}^2 \middle| w \in \mathcal{H}, (\hat{x},x) = \mathbb{K} (w,w^\perp), \ (w,w^\perp) \in \mathcal{H} \right\}.
\]
Hence the contraction principle again, $$\left \{ \left ( f(\hat{X}^{\epsilon}, \cdot  ),f^2(\hat{X}^{\epsilon}, \cdot  ), X^{\epsilon} \right) \right\}_{\epsilon}$$ satisfies the LDP on $C([0,1]) \times C([0,1]) \times C^{\alpha \text{-Hld}}_0([0,1])$ with good rate function
\[
    \tilde{J}^{\#} (a,\tilde{a},x)  = \inf \left \{ \frac{1}{2} \| (w,w^\perp)\|^2_{\mathcal{H}} :   (a,\tilde{a})  = F_f \circ  \mathbb{K}(w,w^\perp) , \ (w,w^\perp) \in \mathcal{H} \right \}.
\] 
and this is the claim.
\end{proof}

\begin{proof}[Proof of Theorem \ref{main5}]
Since $\sigma \in C_b^4$, the coefficient of drift term $\frac{1}{2} ( \sigma^{2} + \sigma \sigma' )$ in (\ref{esrSABR}) is in $C_b^3$.
Then by Lemma \ref{BET} and Proposition \ref{main2} $(i)$, one can apply Theorem \ref{main4} by taking 
$$  ( A^{\epsilon} ,\tilde{A}^{\epsilon}, X^{\epsilon} )  =   ( f(\hat{X}^{\epsilon}, \cdot  ) ,f^2(\hat{X}^{\epsilon}, \cdot  ), X^{\epsilon} )$$
 and the rate function is given by
\begin{align*}
      \tilde{J} (y) 
      &:= \inf \left \{ \tilde{J}^{\#} (a,\tilde{a},x) :  y = \Phi \circ F \left (  a \cdot x, \tilde{a} \cdot \Lambda  \right) , \ x \in \textbf{BV},  \right \}\\
      & = \inf \left \{ \frac{1}{2} \| (w,w^\perp) \|_{\mathcal{H}}^2 :    y = \Phi \circ F \circ  F_f \circ \mathbb{K} (w,w^\perp) , \ (w,w^\perp) \in \mathcal{H} \right\},
\end{align*}
and this is the claim.
\end{proof}

\subsection{Proof of Theorem \ref{short}}\label{shortproof}
\begin{proof}
Since $\{ \epsilon^{\mu} (W_{\epsilon \cdot} ,W^{\perp}_{\epsilon \cdot })\}$ and $\{ \epsilon^{\mu +1/2} (W,W^\perp) \}$ are the same law, one can show that $\{ \epsilon^{\mu} (W_{\epsilon \cdot} ,W^{\perp}_{\epsilon \cdot })\}$ satisfies the LDP on $C([0,1] ) \times C([0,1] ) $ with speed $\epsilon^{2\mu+1}$ with good rate function
\begin{align*}
I^{(0)} (w,w^\perp) := \begin{cases}
\frac{1}{2} \| (w,w^\perp) \|_{\mathcal{H}}^2, & (w,w^\perp ) \in \mathcal{H},\\
\infty, & \text{otherwise}.
\end{cases}
\end{align*}
Since $\| (W_{\epsilon \cdot}, W^\perp_{\epsilon \cdot}) \|_{\alpha \text{-Hld}}$ has a Gaussian tails (see the proof of Proposition \ref{main2} $(ii)$), one can also prove that $\{ \epsilon^{\mu} (W_{\epsilon \cdot} ,W^{\perp}_{\epsilon \cdot })\}$ satisfies the LDP on $C_0^{\alpha\text{-Hld}} ([0,1] ) \times C_0^{\alpha\text{-Hld}} ([0,1] )  $ with speed $\epsilon^{2\mu+1}$ with good rate function $I^{(0)}$.
Let $(W^\epsilon, (W^\perp)^\epsilon ) := ( \epsilon^\mu W_{\epsilon \cdot}, \epsilon^\mu W^\perp_{\epsilon \cdot})$ and $X^\epsilon := \rho W^\epsilon + \sqrt{1-\rho^2} (W^\perp)^\epsilon$.
Then the contraction principle implies that $\{ (X^\epsilon , F( W^\epsilon , \Lambda) ) \}$ satisfies the LDP on $C_0^{\alpha\text{-Hld}} (\mathbb{R}) \times G\Omega^{\alpha\text{-Hld}} (\mathbb{R}^2)  $ with speed $\epsilon^{2\mu+1}$ with good rate function 
\begin{align*}
I^{(1)} (x,X) := \inf \{ \frac{1}{2} \| (w,w^\perp )\|_{\mathcal{H}}^2 : x = \rho  w + \sqrt{1 - \rho^2} w^\perp , \ X = F (w,\Lambda), \ (w,w^\perp) \in \mathcal{H}  \},
\end{align*}
where $F$ is the Young pair, see \eqref{defF}.
Let $\tilde{A}^\epsilon_t := \epsilon^{\mu}( A_{\epsilon t} - A_0 )$.
By using the change of variable for stochastic integrals and Riemann-integrals, one can prove that $\tilde{A}^\epsilon$ is the solution of the following It\^{o} SDE:
\begin{align*}
\tilde{A}^\epsilon_t = \int^t_0 \tilde{a}^\epsilon (\tilde{A}^\epsilon_u) \D ( \epsilon^\mu W_{\epsilon u} )+ \int_0^t \tilde{b}^\epsilon (\tilde{A}^\epsilon_u) \D u,
\end{align*}
where 
\begin{align*}
\tilde{a}^\epsilon (y) := a (A_0 + \epsilon^{-\mu} y) , \quad \tilde{b}^\epsilon(y) := \epsilon^{1+\mu} b(A_0 + \epsilon^{-\mu} y).
\end{align*}
Then one can show that $\tilde{A}^\epsilon$ is the solution of the following Stratonovich SDE,
\begin{align*}
\tilde{A}^\epsilon_t = \int^t_0 \tilde{a}^\epsilon (\tilde{A}^\epsilon_u) \circ  \D ( \epsilon^\mu W_{\epsilon u} ) 
- \frac{1}{2} \int^t_0  \tilde{a}^\epsilon  (\tilde{a}^\epsilon)'  (\tilde{A}^\epsilon_u) \D ( \epsilon^{2\mu+1} u ) 
+ \int_0^t \tilde{b}^\epsilon (\tilde{A}^\epsilon_u) \D u,
\end{align*}
and by Theorem \ref{thm:33}, $\tilde{A}^\epsilon $ is the solution of RDE with a coefficient $( \tilde{a}^\epsilon, - \frac{\epsilon^{2\mu+1}}{2} \tilde{a}^\epsilon  (\tilde{a}^\epsilon)' +\tilde{b}^\epsilon   )$. 

Let $\tilde{\Phi}_\epsilon$ be the solution map of RDE with the coefficient $( \tilde{a}^\epsilon, - \frac{\epsilon^{2\mu+1}}{2} \tilde{a}^\epsilon  (\tilde{a}^\epsilon)' +\tilde{b}^\epsilon   )$ i.e. $\tilde{A}^\epsilon = \tilde{\Phi}_\epsilon \circ F (W,\Lambda)$.
Note that 
\[
 \| \tilde{a}^\epsilon \|_{C^3_b} \lesssim \|a\|_{C^3_b}, \quad \| -  ( \epsilon^{2\mu+1}/2 ) \tilde{a}^\epsilon  (\tilde{a}^\epsilon)' +  \tilde{b}^\epsilon/2 \|_{C^3_b} \lesssim ( \|a\|_{C^4_b}+\|b\|_{C^3_b},
\]
where the proportional constant does not depend on $\epsilon$.
Let $\tilde{\Phi_0} $ be the solution map of RDE with the coefficient $(a(A_0), 0)$.
Since the upper bound of $\|\cdot \|_{C^3_b}$ norm for the coefficient $( \tilde{a}^\epsilon, - \frac{\epsilon^{2\mu+1}}{2} \tilde{a}^\epsilon  (\tilde{a}^\epsilon)' +\tilde{b}^\epsilon   )$ is $\epsilon$-uniform,   we can show that $\{\tilde{\Phi}_\epsilon \}_{\epsilon}$ is equicontinuous, and for any $(x,X) \in C^{\alpha \text{-Hld}}_0 ([0,1]) \times  G\Omega_{\alpha \text{-Hld}} (\bbR^2)$ with $I^{(1)} (x,X) < \infty$, $\tilde{\Phi}_\epsilon (X) \to \tilde{\Phi} (X)$. 
Then we have that for any converging sequence $(x_\epsilon, X_\epsilon) \to (x,X)$ with $I^{(1)} (x,X) < \infty$, $(x_\epsilon , \tilde{\Phi}_\epsilon (X_\epsilon))$ converges to $(x, \tilde{\Phi}_0 (X))$, and so the extended contraction principle (Theorem 2.1 in \cite{Pu}) implies that $\{ (X^\epsilon, \tilde{A}^\epsilon)\}$ satisfies the LDP on $C^{\alpha\text{-Hld}} (\bbR^2)$ with speed $\epsilon^{2\mu+1}$ with good rate function 
 \begin{align*}
I^{(2)} (x,\tilde{a}) := \inf \left \{  \frac{1}{2} \| (w ,w^\perp )\|_{\mathcal{H}}^2 : x = \rho  w + \sqrt{1 - \rho^2} w^\perp,\ \tilde{a} = a (A_0) w \right \}.
\end{align*}
Since $\{ \mathcal{K}_\epsilon \}_{\epsilon >0}$ is equicontinuous and converge to the usual fractional kernel $\mathcal{K}_0$ (see Appendix \ref{proK}), the extended contraction principle implies that $\{ (X^\epsilon , V^\epsilon ) \}$ satisfies the LDP on $C^{\alpha\text{-Hld}} (\bbR) \times C(\bbR)$ with speed $\epsilon^{2\mu+1}$ with good rate function 
 \begin{align*}
I^{(3)} (x,v) := \inf \left  \{ \frac{1}{2} \| (w,w^\perp )\|_{\mathcal{H}}^2 : x = \rho  w + \sqrt{1 - \rho^2} w^\perp, \  v = \Psi \mathcal{K}_0 (a (A_0) w ) \right \},
\end{align*}

By using the assumption of $f$, and Proposition \ref{main2} $(ii)$, we can apply the contraction principle and Theorem \ref{main1}, and so $\{ F( Z^\epsilon) := F ((Z^{(1)})^\epsilon , (Z^{(2)})^\epsilon ) \}$ satisfies the LDP on $G \Omega_{\alpha\text{-Hld}}(\bbR^2) $ with speed $\epsilon^{2\mu+1}$ with good rate function 
 \begin{align*}
I^{(4)} (X) 
:= \inf \left \{  \frac{1}{2} \| (w,w^\perp ) \|_{\mathcal{H}}^2 :
  \begin{aligned}
  & x = \rho  w + \sqrt{1 - \rho^2} w^\perp ,\\
    & X = F  (f(\Psi \mathcal{K}_0 (a (A_0) w ),0) \cdot x, 0 ) 
      \end{aligned}  \right\},
\end{align*}

Let $\Phi_{\epsilon}$ be the solution map of RDE with the coefficient $( \tilde{\sigma}^\epsilon, - \frac{1}{2} \{  \tilde{\sigma }^\epsilon  (\tilde{\sigma}^\epsilon)' + (\tilde{\sigma}^\epsilon)^2 \}  )$ i.e. $\tilde{Y}^\epsilon = \Phi_\epsilon \circ F (Z^{\epsilon} )$. 
Since the same reason as $\tilde{\Phi}_\epsilon$,  $\{\Phi_{\epsilon} \}_{\epsilon \in (0,1]}$ is equcontinuous, and one can prove that for any sequence with $ X_\epsilon \to X$ with $I^{(4)}(X) < \infty$, $ \tilde{\Phi}_\epsilon (X_\epsilon) \to  \tilde{\Phi}_0 (X)$.
Therefore,  the extended contraction principle~\cite{Pu} implies that $\{ \tilde{Y}^\epsilon\}$ satisfies the LDP on $C^{\alpha\text{-Hld}}$ with speed $\epsilon^{2\mu+1}$ with good rate function
 \begin{align*}
 \tilde{J} (\tilde{y}) 
:= \inf \left \{ \frac{1}{2} \{ (w,w^\perp) \|_{\mathcal{H}}^2 :
  \begin{aligned}
  & x = \rho  w + \sqrt{1 - \rho^2} w^\perp,\\
    & \tilde{y} = \sigma(y_0) \int_0^{\cdot} f(\Psi \mathcal{K}_0 (a (A_0) w )_r,0) \D x_r  
      \end{aligned}  \right\},
\end{align*}
and this is the claim.
\end{proof}

\appendix
\section{Some properties for $\mathcal{K}^\epsilon$}\label{proK}

\begin{proposition}
We fix $\alpha, \gamma \in (0,1)$. Under the Hypothesis \ref{hypo}, we have the following:
\begin{enumerate}
\item $\{ \mathcal{K}^\epsilon\}_{\epsilon \in (0,1]}$ is equicontinuous,
\item for all $f \in C^\alpha$, $\mathcal{K}^\epsilon f$ converges to $\mathcal{K}_0f$. 
\end{enumerate}
\end{proposition}

\begin{proof}
$(i)$ Since $\mathcal{K}^\epsilon$ is linear, it is enough to show that for any $f \in C^{\alpha\text{-Hld}}$, there exists a constant $C>0$ (uniformly $\epsilon$ and $f$) such that $\|\mathcal{K}^\epsilon f \|_{\gamma} \le C \|f\|_{\alpha}$.
 First note that 
\begin{align*}
|\kappa_\epsilon (t) |  \lesssim \epsilon^\mu t^\mu, \quad 
\left  |\frac{\D}{\D t}  \kappa_\epsilon (t) \right  |  \lesssim \epsilon^\mu t^{\mu-1}, \quad 
\left  | \frac{\D^2}{\D t^2} \kappa_\epsilon (t) \right | \lesssim \epsilon^\mu t^{\mu -2}.
\end{align*}
We will estimate the first term. Let $ \phi_\epsilon(t) := \epsilon^{-\mu}  \kappa_{\epsilon} (t) (f(t) - f(0)) $.
 For $\forall t \in [0,1]$ and $ \forall h \in (0,1-t]$, we have that 
\begin{align*}
\left| \phi_\epsilon(t+h) - \phi_\epsilon(t)  \right| 
& \le \epsilon^{-\mu} \left\{  |\kappa_\epsilon(t+h) | f(t+h) -f(t)| +  |f(t) -f(0)| |\kappa_\epsilon(t+h) -\kappa_\epsilon(t) |   \right\}\\
&  \le   \| f \|_{\alpha } h^{\alpha} (t+h)^{\mu} +\epsilon^{-\mu} \|f\|_{\alpha} t^{\alpha } \left ( \int_t^{t+h}  \left| \frac{\D}{\D r} \kappa_\epsilon (r)  \right| \D r \right )  \\
& \lesssim \|f\|_{\alpha} \left \{   |h|^\gamma + t^{\alpha} (t^\mu - (t+h)^\mu)   \right\} \lesssim \|f\|_{\alpha}  |h|^\gamma.
\end{align*}
Here we use that 
\begin{align}\label{Mu}
 t^{\gamma} (t+h)^{\mu} ( (t+h)^{-\mu} - t^{-\mu}) \lesssim h^{\gamma},
 \end{align}
  in the final inequality (see \cite{Mu}[Chapter1, Page 15]).
The case $h <0$ is analogous. 

We now consider the second term. Let 
\[
\varphi_\epsilon (t) := \epsilon^{-\mu}  \int_0^t (f(t)- f(s) ) \frac{\D}{\D t} \kappa_{\epsilon} (t-s) \D s 
= \epsilon^{-\mu}  \int_0^t (f(t)- f(t-r) ) \frac{\D}{\D r} \kappa_{\epsilon} (r) \D r.
\]
Then the change of variables implies that 
\begin{align*}
\varphi_\epsilon (t +h) - \varphi_\epsilon (t) 
&= \epsilon^{-\mu} \int_0^t (f(t)- f(t-r) ) \left ( \frac{\D}{\D r} \kappa_{\epsilon} (r+h) -  \frac{\D}{\D r} \kappa_{\epsilon} (r) \right ) \D r\\
& \quad + \epsilon^{-\mu} \int_0^t (f(t+h)- f(t) )  \frac{\D}{\D r} \kappa_{\epsilon} (r+h) \D r \\
& \quad + \epsilon^{-\mu} \int_{-h}^0 (f(t+h)- f(t-r) )  \frac{\D}{\D r} \kappa_{\epsilon} (r+h) \D r =: I_1+ I_2+I_3.
\end{align*}
Then we have that 
\begin{align*}
|I_1| & \le \epsilon^{-\mu} \int_0^t \int_r^{r+h} |f(t) - f(t-r)| \left |\frac{\D^2}{ \D u^2} \kappa_\epsilon (u) \right| \D u \D r\\
& \le \|f\|_{\alpha} \int_0^t \int_r^{r+h} r^{\alpha} u^{\mu-2} \D u \D r\\
&   \lesssim \|f\|_{\alpha } |h|^\gamma \left(  \int_0^{t/h} r^\alpha [r^{\mu-1} - (1+r)^{\mu-1}] \D r \right) \lesssim \|f\|_{\alpha } |h|^\gamma.
\end{align*}
Here we use that the function $ y \mapsto 1- (1+y)^{\mu -1} + (\mu -1) y$ is concave with a maximum value of $0$ at $y=0$ in the last inequality.
One can also obtain that 
\begin{align*}
|I_2|   \le \epsilon^{-\mu} \|f\|_\alpha \int_0^t h^\alpha \left| \frac{\D}{ \D r} \kappa_\epsilon (r+h) \right| \D r 
 \le \|f\|_{\alpha} h^\alpha \int_0^t (r+h)^{\mu-1} \D r 
 \lesssim \|f\|_\alpha h^\gamma 
\end{align*}
\begin{align*}
|I_3| \le \epsilon^{-\mu} \|f\|_\alpha \int_{-h}^0 (r+h)^\alpha \left| \frac{\D}{ \D r} \kappa_\epsilon (r+h) \right| \D r
\le  \|f\|_\alpha \int_{-h}^0  (u+h)^{\alpha + \mu -1} \D r 
\lesssim \|f\|_\alpha h^\gamma
\end{align*}
and these inequalities imply the first assertion.

$(ii)$ The simple calculation implies that 
\begin{align*}
& |\mcK^\epsilon f(t) - \mcK_0 f(t) - \mcK^\epsilon f(s) + \mcK_0 f(s) | \\
& \le  \left | \phi_\epsilon (t) -  (f(t) -f(0))  t^\mu -  \phi_\epsilon (s)  +  (f(s) -f(0))  s^\mu \right| \\
& \quad + \Big | \varphi_\epsilon (t)  - \alpha \int^t_0 (f(t) - f(u) ) (t-u)^{\mu-1} \D u  -  \varphi_\epsilon (s) + \alpha \int^s_0 (f(s) - f(u) ) (s-u)^{\mu-1} \D u \Big | \\
& \le \left |  \epsilon^{-\mu} (f(t) -f(s) ) \kappa_\epsilon (t) - (f(t) -f(s)) t^\mu \right|  \\
& \quad + \left|\epsilon^{-\mu} (f(s) -f(0) ) ( \kappa_\epsilon (t) -  \kappa_\epsilon (s) ) - (f(s) -f(0)) ( t^\mu - s^\mu) \right |\\
& \quad + \left| \int_0^s (f(s) - f(t)) \left ( \epsilon^{-\mu} \frac{\D}{\D t} \kappa_\epsilon (t-r) - \mu (t-r)^{\mu-1} \right) \D r \right|\\
& \quad + \left |  \int_0^s (f(r) -f(s)) \left( \epsilon^{-\mu} \frac{\D}{\D t} \kappa_\epsilon(t-r) - \mu (t-r)^{\mu-1}  - \epsilon^{-\mu}\frac{\D}{\D s} \kappa_\epsilon(s-r) +\mu (s-r)^{\mu-1} \right) \D r \right|\\
& \quad + \left|  \int_s^t (f(r) - f(t))  \left( \epsilon^{-\mu} \frac{\D}{\D t}  \kappa_\epsilon(t-r)  - \mu (t -r)^{\mu-1} \right) \D r \right|  = : T_1 + T_2 + T_3 + T_4 + T_5,
\end{align*}
and so we will estimate them. $T_1$ is simply estimated by 
\begin{align*}
|T_{1}| 
& \le \|f\|_{\alpha} |t-s|^\alpha | \epsilon^{-\mu} \kappa_\epsilon (t) -  t^\mu  | \\
& \le \|f\|_\alpha |t-s|^\gamma \left( \frac{|t-s|}{t} \right)^{-\mu} \left(  \sup_{t \in[0,1]} |g(\epsilon t) -1 | \right) \lesssim  \|f\|_\alpha |t-s|^\gamma  \left(  \sup_{t \in[0,1]} |g(\epsilon t) -1 | \right).
\end{align*}
We also have that 
\begin{align*}
|T_{2}| 
& = |(f(s)- f(0)) \{  (g(\epsilon t) - g(\epsilon s) ) t^\mu + (g(\epsilon s) -1) (t^\mu - s^\mu) \}|\\
& \le |f(s) -f(0) | \left| \int_{\epsilon s}^{\epsilon t} \frac{\D}{\D r} g(r) \D r   \right| t^{\mu} + |f(s) -f(0) | | g(\epsilon s) -1|  |t^\mu - s^\mu|\\
& \le \|f\|_\alpha \left \{ \epsilon   |t-s|  s^\alpha t^\mu  + s^\alpha  \left(  \sup_{t \in[0,1]} |g(\epsilon t) -1 | \right)  |t^\mu - s^\mu| \right\}\\
& \le \|f\|_\alpha \left \{  \epsilon |t-s|^\gamma +   \left(  \sup_{t \in[0,1]} |g(\epsilon t) -1 | \right) s^\alpha  ( s^\mu - t^\mu) \right \}\\
& \lesssim \|f\|_\alpha \left ( \epsilon +  \sup_{t \in[0,1]} |g(\epsilon t) -1 | \right ) |t-s|^\gamma.
\end{align*}
Here we use \eqref{Mu} in the last inequality. $T_3$ is estimated by
\begin{align*}
|T_{3} | 
& \le \|f\|_{\alpha} |t-s|^\alpha  \int_0^s \left|  \epsilon^{-\mu} \frac{\D}{\D t} \kappa_\epsilon (t-r) - \mu (t-r)^{\mu-1} \right |  \D r \\
& \le \|f\|_{\alpha} |t-s|^\alpha  \int_0^s \left|   \epsilon g' (\epsilon (t-r)) (t-r)^\mu + \mu g(\epsilon (t-r)) (t-r)^{\mu-1} - \mu (t-r)^{\mu-1}  \right|  \D r \\
& \le \|f\|_\alpha |t-s|^\alpha \left\{ \epsilon \int_0^s (t-r)^\mu \D r + |\mu| \int_0^s |g(\epsilon (t-r)) -1| (t-r)^{\mu-1} \D r \right\}\\
& \lesssim \|f\|_\alpha |t-s|^\gamma \left ( \epsilon +  \sup_{t \in[0,1]} |g(\epsilon t) -1 | \right ).
\end{align*}
 To estimate $T_4$, we decompose it as follows: 
\begin{align*}
|T_4| & \le \|f\|_{\alpha}  \int_0^s (s-r)^\alpha \Bigg \{  \left| \epsilon g'(\epsilon (t-r)) (t-r)^\mu - \epsilon g'(\epsilon (s-r)) (s-r)^\mu \right| \\
& \qquad + |\mu|   \left |   g(\epsilon (t-r)) (t-r)^{\mu-1} -  (t-r)^{\mu-1}   -  g(\epsilon (s-r)) (s-r)^{\mu-1} + (s-r)^{\mu-1}  \right| \Bigg \} \D r  \\
& \lesssim  \int_0^s (s-r)^\alpha    \left| \epsilon g'(\epsilon (t-r)) (t-r)^\mu - \epsilon g'(\epsilon (s-r)) (s-r)^\mu \right| \D r \\
& \quad +  \int_0^s (s-r)^\alpha |g(\epsilon(t-r)) - g(\epsilon(s-r))| |(t-r)^{\mu-1} - (s-r)^{\mu-1}| \D r \\
& \quad +    \int_0^s (s-r)^{\gamma-1} |g(\epsilon(t-r)) - g(\epsilon(s-r))|   \D r \\
& \quad +   \int_0^s (s-r)^\alpha |g(\epsilon(s-r)) - 1 | |(t-r)^{\mu-1} - (s-r)^{\mu-1}| \D r  =: T_{41} + T_{42} + T_{43} + T_{44},
\end{align*} 
The estimations of them are obtained as follows:
\begin{align*}
|T_{41}| 
& = \epsilon \int_0^s r^\alpha \left| \{ g'(\epsilon (t-s+r)) - g'(\epsilon r) \} (t-s +r)^\mu + g'(\epsilon r) ((t-s+r)^\mu -r^\mu ) \right| \D r \\
& \le \epsilon \int_0^s r^\alpha \left | \left( \int_{\epsilon r}^{\epsilon (t-s+r)} g''(u) \D u \right) (t-s +r)^\mu  \right| \D r  + \epsilon \int_0^s r^\alpha |g'(\epsilon r)| |(t-s+r)^\mu - r^\mu| \D r\\
& \le \epsilon^2  |t-s| \int_0^s r^\alpha  (t-s +r)^\mu \D r + \epsilon \int_0^s r^\alpha  (r^\mu - (t-s+r)^\mu ) \D r  \\
& \le \epsilon^2  |t-s|^{1+\mu} \left( \int_0^s r^\alpha \D r \right) + \epsilon |t-s|^\gamma  \int_0^s \D r\\
& \lesssim ( \epsilon + \epsilon^2 )    |t-s|^{\gamma}
\end{align*}

\begin{align*}
|T_{42}| & \le \int_0^s  (s-r)^\alpha \left | \int_{\epsilon (s-r) }^{\epsilon (t-r)} g'(u)\D u \right|  |(t-r)^{\mu-1} - (s-r)^{\mu-1} | \D r\\
& \le \epsilon |t-s| \int_0^s r^\alpha   |r^{\mu-1} - (t-s+r)^{\mu-1} | \D r \lesssim \epsilon  |t-s|^{\gamma},
\end{align*}
Here we use the same argument as the estimation of $I_1$ in $(i)$ in the final equality.
\begin{align*}
|T_{43}| 
& \le \int_0^s (s-r)^{\gamma-1} \left|  \int_{\epsilon (s-r) }^{\epsilon (t-r)} g'(u)\D u \right| \D r  \le \epsilon |t-s| \int_0^s (s-r)^{\gamma-1} \D r \lesssim \epsilon  |t-s|^{\gamma},
\end{align*}

\begin{align*}
|T_{44}| 
& \le \left (  \sup_{t \in[0,1]} |g(\epsilon t) -1 | \right )  \int_0^s (s-r)^\alpha  | (t-r)^{\mu-1} -  (s-r)^{\mu-1} | \D r \\
& \le  \left (  \sup_{t \in[0,1]} |g(\epsilon t) -1 | \right )  \int_0^s r^\alpha  | r^{\mu-1} -  (t-s+r)^{\mu-1} | \D r 
 \lesssim \epsilon  |t-s|^{\gamma},
\end{align*}
Here we use again the same argument as the estimation of $I_1$ in $(i)$ in the final equality. 
These inequalities imply the desired estimate of $T_4$. 
Finally, one has that 
\begin{align*}
|T_5|
& \le \int_s^t |f(t) - f(r)| \left| \epsilon g'(\epsilon (t-r)) (t-r)^\mu + \mu g(\epsilon(t-r)) (t-r)^{\mu-1} - \mu (t-r)^{\mu-1} \right| \D r\\
& \le \|f\|_\alpha \left\{ \epsilon \int_s^t (t-r)^\gamma \D r + |\mu| \int_s^t  (t-r)^{\gamma-1} |g(\epsilon(t-r)) -1| \D r  \right\} \\
& \lesssim \epsilon|t-s|^{\gamma+1} + \left (  \sup_{t \in[0,1]} |g(\epsilon t) -1 | \right ) \int_s^t (t-r)^{\gamma-1} \D r\\
& \le \left( \epsilon +  \sup_{t \in[0,1]} |g(\epsilon t) -1 |   \right) |t-s|^\gamma
\end{align*}
and so we have the claim.
\end{proof}

\section{Proof of Theorem \ref{putcall}}\label{B}

\begin{proof}
{\noindent  (i) Lower bound}

For $x \le 0$ and $\delta >0$, one has that 
\begin{align*}
\EXP{(\exp{(x t^{-\mu})} -S_t)_+ } 
& \ge \EXP{1_{\{ \exp{(x(1+\delta)t^{-\mu })} >S_t \}} \left( \exp{(xt^{-\mu})} - S_t \right)}\\
& \ge \left ( \exp{(xt^{-\mu})} - \exp{(x (1+\delta ) t^{-\mu})}\right) \PROB{ \exp{(x (1+\delta ) t^{-\mu})} > S_t}\\
& \ge \exp{(x (1+\delta ) t^{-\mu})} (- x \delta t^{-\mu} ) \PROB{ \exp{(x (1+\delta ) t^{-\mu})} > S_t}.
\end{align*}
Since $\lim_{t \searrow 0} t^{2\mu +1} \log t =0$, one has that 
\begin{align*}
& \liminf_{t \searrow 0} t^{2\mu +1} \log \EXP{(\exp{(x t^{-\mu})} -S_t)_+ } \\
&  \ge \liminf_{t \searrow 0} t^{2 \mu+1} \left ( x (1+\delta) t^{-\mu} + \log(-x) + \log \delta + -\mu \log t + \log \PROB{ \exp{(x (1+\delta ) t^{-\mu})} > S_t}  \right) \\
& = \liminf_{t \searrow 0} t^{2 \mu+1}   \log \PROB{ \exp{(x (1+\delta ) t^{-\mu})} > S_t}
\end{align*}
and so Theorem \ref{shorttime} implies that 
\begin{align*}
\liminf_{t \searrow 0} t^{2\mu +1} \log \EXP{(\exp{(x t^{-\mu})} -S_t)_+ }  \ge - \Lambda^\ast (x (1+\delta)),
\end{align*}
and so the continuity of $\Lambda^\ast$ implies the lower bound.

{\noindent  (ii) Upper bound.}

For all $q >1$, the \Hd \ inequality implies that 
\begin{align*}
\EXP{(\exp{(x t^{-\mu})} -S_t)_+} 
&= \EXP{(\exp{(x t^{-\mu})} -S_t)_+ 1_{\{ \exp{(x t^{-\mu})} >S_t \}} } \\
& \le \EXP{(\exp{(x t^{-\mu})} -S_t)^q_+ }^{1/q} \EXP{ 1_{\{ \exp{(x t^{-\mu})} >S_t \}} }^{1-1/q}
 \end{align*}
 and so
 \begin{align*}
& t^{2\mu +1} \log \EXP{(\exp{(x t^{-\mu})} -S_t)_+} \\
 & \le  \frac{t^{2\mu+1}}{q} \log \EXP{(\exp{(x t^{-\mu})} -S_t)^q_+ } + t^{2\mu+1} (1-1/q ) \PROB{ 1_{\{ \exp{(x t^{-\mu})} >S_t \}}} = : T^{(1)}_t+T^{(2)}_t.
 \end{align*}
 Since $S$ is positive, 
 \begin{align}\label{upper}
 \limsup_{t\searrow 0} T^{(1)}_t \le  \limsup_{t\searrow 0}  \frac{t^{2\mu+1}}{q} \log \EXP{\exp{( q x t^{-\mu})}  } =  \limsup_{t\searrow 0}  x t^{\mu+1} = 0.
 \end{align}
 By Theorem \ref{shorttime}, it is also true that 
 \begin{align*}
  \limsup_{t\searrow 0} T^{(2)}_t \le  - \Lambda^\ast (x).
 \end{align*}
 Combined with above two inequalities one can obtain the upper bound. 
 By $(i)$, $(ii)$, one can obtain the first assertion. 
 
 In the second assertion, one has to improve the estimate of \ref{upper} because $(\cdot- \exp{(x t^{-\mu})})_+$ is not bounded.
By using the assumption \eqref{moment}, one can estimate $\EXP{S_t^q}$ instead of $\EXP{\exp{( q x t^{-\mu})}  }$ and the same argument as above implies the required bound. 
\end{proof}

\end{document}